\documentclass{amsart}
\usepackage{amsmath}
\usepackage{amssymb}
\usepackage{amsfonts}

\newtheorem{theorem}{Theorem}
\theoremstyle{plain}

\newtheorem{corollary}{Corollary}

\newtheorem{example}{Example}

\newtheorem{lemma}{Lemma}

\newtheorem{proposition}{Proposition}
\newtheorem{remark}{Remark}

\numberwithin{equation}{section}
\input{tcilatex}

\begin{document}
\title{Almost $\alpha $-Cosymplectic $(\kappa ,\mu ,\nu )$-Spaces}
\author{ Hakan \"{O}zt\"{u}rk, Nesip Aktan}
\address[ Hakan \"{O}zt\"{u}rk and Nesip Aktan ]{ Afyonkarahisar Kocatepe
University, Faculty\ of\ Art\ and\ Sciences, Department\ of\ Mathematics,
Afyonkarahisar/TURKEY}
\email{hozturk@aku.edu.tr, naktan@aku.edu.tr}
\author{Cengizhan Murathan}
\address[Cengizhan Murathan]{ Uluda\u{g} University, Faculty\ of\ Art\ and\
Sciences, Department\ of\ Mathematics, Bursa /TURKEY}
\email{cengiz@uludag.edu.tr}
\subjclass[2000]{53D10, 53C25, 53C35.}
\keywords{Almost $\alpha $-cosymplectic manifold, $(\kappa ,\mu ,\nu )$%
-space, distribution.}

\begin{abstract}
Main interest of the present paper is to investigate the almost $\alpha $%
-cosymplectic manifolds for which the characteristic vector field of the
almost $\alpha $-cosymplectic structure satisfies a specific $(\kappa ,\mu
,\nu )$-nullity condition. This condition is invariant under $D$-homothetic
deformation of the almost cosymplectic $(\kappa ,\mu ,\nu )$-spaces in all
dimensions. Also, we prove that for dimensions greater than three, $\kappa
,\mu ,\nu $ are not necessary constant smooth functions such that $df\wedge
\eta =0.$ Then the existence of the three-dimensional case of almost
cosymplectic $(\kappa ,\mu ,\nu )$-spaces are studied. Finally, we construct
an appropriate example of such manifolds.
\end{abstract}

\maketitle

\section{Introduction}

It is well known that there exist contact metric manifolds $(M^{2n+1},\phi
,\xi ,\eta ,g),$ for which the curvature tensor $R$ and the direction of the
characteristic vector field $\xi $ satisfy $R(X,Y)\xi =0,$ for any vector
fields on $M^{2n+1}.$ Using a $D$-homothetic deformation to a contact metric
manifold with $R(X,Y)\xi =0$ we get a contact metric manifold satisfying the
following special condition%
\begin{equation}
R(X,Y)\xi =\eta (Y)(\kappa I+\mu h)X-\eta (X)(\kappa I+\mu h)Y,  \label{500}
\end{equation}%
where $\kappa ,\mu $ are constants and $h$ is the self-adjoint $(1,1)$%
-tensor field. This condition is called $(\kappa ,\mu )$-nullity on $%
M^{2n+1}.$ Contact metric manifolds with $(\kappa ,\mu )$-nullity condition
studied for $\kappa ,\mu =$const. in (\cite{bo1}, \cite{blr}).

In \cite{bo1}, the author introduced contact metric manifold whose
characteristic vector field belongs to the $(\kappa ,\mu )$-nullity
condition and proved that non-Sasakian contact metric manifold is completely
determined locally by its dimension for the constant values of $\kappa $ and 
$\mu .$

Koufogiorgos and Tsichlias found a new class of $3$-dimensional contact
metric manifolds that $\kappa $ and $\mu $ are non-constant smooth
functions. They generialized $(\kappa ,\mu )$-contact metric manifolds $%
(M^{2n+1},\phi ,\xi ,\eta ,g)$ for dimensions greater than three on
non-Sasakian manifolds, where the functions $\kappa ,\mu $ are constant.

Following these works, P.Dacko and Z.Olszak extensively have studied almost
cosymplectic $(\kappa ,\mu ,\nu )$ manifolds. These almost cosymplectic
manifolds\ whose almost cosymplectic structures $(\phi ,\xi ,\eta ,g)$
satisfy the condition%
\begin{equation}
R(X,Y)\xi =\eta (Y)(\kappa I+\mu h+\nu \phi h)X-\eta (X)(\kappa I+\mu h+\nu
\phi h)Y,  \label{501}
\end{equation}%
for $\kappa ,\mu ,\nu \in \mathcal{R}_{\eta }(M^{2n+1}),$ where $\mathcal{R}%
_{\eta }(M^{2n+1})$ be the subring of the ring of smooth functions $f$ on $%
M^{2n+1}$ for which $df\wedge \eta =0.$ In the sequel, such manifolds are
called almost cosymplectic $(\kappa ,\mu ,\nu )$-spaces (\cite{olszak3}).
The condition (\ref{501}) is invariant with respect to the $D$-homothetic
deformations of those structures. The authors showed that the integral
submanifolds of the distribution $\mathcal{D}$ of such manifolds were
locally flat Kaehlerian manifolds and found a new characterization which was
established up to a $D$-homothetic deformation of the almost cosymplectic
manifolds.In \cite{dacko olzsak}, a complete local description of almost
cosymplectic $(-1,%
%TCIMACRO{\U{3bc} }%
%BeginExpansion
\mu
%EndExpansion
,0)$-spaces via \textquotedblleft model spaces\textquotedblright\ is given,
depending on the function $%
%TCIMACRO{\U{3bc} }%
%BeginExpansion
\mu
%EndExpansion
$. When $%
%TCIMACRO{\U{3bc} }%
%BeginExpansion
\mu
%EndExpansion
$ is constant, the models are Lie groups with a left-invariant almost
cosymplectic structure.

Moreover, Pastore and Dileo are studied the curvature properties of almost
Kenmotsu manifolds, with special attention to $(\kappa ,\mu )$-nullity
condition for $\kappa ,\mu =$const. and $\nu =0$ (see \cite{pastore2}, \cite%
{pastore3}). In \cite{pastore2} the authors prove that an almost Kenmotsu
manifold $M^{2n+1}$ is locally a warped product of an almost Kaehler
manifold and an open interval. If additionally $M^{2n+1}$ is locally
symmetric then it is locally isometric to the hyperbolic space $\mathbf{H}%
^{2n+1}$ of constant sectional curvature $c=-1$. It is recall that model
spaces for almost cosymplectic case were given in (\cite{olszak3}). But we
did not know any example of an almost $\alpha $-Kenmotsu manifold satisfying
(\ref{501}) with non-constant smooth functions. The following question
sounds that especially interesting. Do there exist almost $\alpha $%
-cosymplectic manifolds satisfying (\ref{501}) with $\kappa ,\mu $
non-constant smooth functions? In this paper, we will try to give an answer
to this question for dimension $3.$

The existence and invariance of the condition (\ref{501}) have been our
motivation in studying almost $\alpha $-cosymplectic manifold.

Section $2$ is devoted to preliminaries on almost contact metric structures.
In section $3$ we give the concept of almost $\alpha $-cosymplectic
manifolds, state general curvature properties and derive several important
formulas on almost $\alpha $-cosymplectic manifolds. These formulas enable
us to find the geometrical properties of almost $\alpha $-cosymplectic
manifolds with $\eta $-parallel tensor $h.$ In section $4$ we study almost $%
\alpha $-cosymplectic manifolds with $\eta $-parallel tensor field $h$ under
some certain conditions and prove that the integral submanifolds of the
distribution $\mathcal{D}$ have Kaehlerian structures if $h$ is $\eta $%
-parallel. In section $5$ we introduce the notion of almost $\alpha $%
-cosymplectic $(\kappa ,\mu ,\nu )$-spaces in terms of a specific curvature
condition. We give a characterization of a $D$-homothetic deformation of the
almost cosymplectic $(\kappa ,\mu ,\nu )$-spaces. We prove that for
dimensions greater than three, $\kappa ,\mu ,\nu $ are not necessary
constant smooth functions. Also, we will prove why the functions $\kappa
,\mu $ and $\nu $ are element of $\mathcal{R}_{\eta }(M^{2n+1}).$

Finally, in section $6$ we investigate the existence of the
three-dimensional case of almost cosymplectic $(\kappa ,\mu ,\nu )$-spaces.
Then we give an example and describe the three-dimensional case.

\section{Almost $\protect\alpha $-cosymplectic manifolds}

An almost contact manifold is an odd-dimensional manifold $M^{2n+1}$ which
carries a field $\phi $ of endomorphisms of the tangent spaces, a vector
field $\xi $, called characteristic or Reeb vector field , and a 1-form $%
\eta $ satisfying $\phi ^{2}=-I+\eta \otimes \xi $ and $\eta (\xi )=1$,
where $I:TM^{2n+1}\rightarrow TM^{2n+1}$ is the identity mapping. From the
definition it follows also that $\phi \xi =0,$ $\eta \circ \phi =0$ and that
the $(1,1)$-tensor field $\phi $ has constant rank $2n$ (see \cite{blair}).
An almost contact manifold $(M^{2n+1},\phi ,\xi ,\eta )$ is said to be
normal when the tensor field $N=[\phi ,\phi ]+2d\eta \otimes \xi $ vanishes
identically, $[\phi ,\phi ]$ denoting the Nijenhuis tensor of $\phi .$ It is
known that any almost contact manifold $(M^{2n+1},\phi ,\xi ,\eta )$ admits
a Riemannian metric $g$ such that%
\begin{equation}
\begin{array}{c}
g(\phi X,\phi Y)=g(X,Y)-\eta \left( X\right) \eta \left( Y\right) ,%
\end{array}
\label{*}
\end{equation}%
for any vector fields $X,Y$ on $M^{2n+1}$. This metric g is called a
compatible metric and the manifold $M^{2n+1}$ together with the structure $%
(M^{2n+1},\phi ,\xi ,\eta ,g)$ is called an almost contact metric manifold.
As an immediate consequence of (\ref{*}), one has $\eta =g(.,\xi ).$ The $2$%
-form $\Phi $ of $M^{2n+1}$ defined by $\Phi (X,Y)=g(\phi X,Y),$ is called
the fundamental 2-form of the almost contact metric manifold $M^{2n+1}$.
Almost contact metric manifolds such that both $\eta $ and $\Phi $ are
closed are called almost cosymplectic manifolds and almost contact metric
manifolds such that $d\eta =0$ and $d\Phi =2\eta \wedge \Phi $ are almost
Kenmotsu manifolds. Finally, a normal almost cosymplectic manifold is called
a cosymplectic manifold and a normal almost Kenmotsu manifolds is called
Kenmotsu manifold.

An almost contact metric manifold $M^{2n+1}$ is said to be almost $\alpha $%
-Kenmotsu if $d\eta =0$ and $d\Phi =2\alpha \eta \wedge \Phi ,$ $\alpha $
being a non-zero real constant. Geometrical properties and examples of
almost $\alpha $-Kenmotsu manifolds are studied in \cite{kimpak}, \cite%
{vaisman}, \cite{kenmotsu} and \cite{olszak}. Given an almost Kenmotsu
metric structure $(\phi ,\xi ,\eta ,g)$, consider the deformed structure%
\begin{equation*}
\eta ^{\shortmid }=\frac{1}{\alpha }\eta ,~~~\xi ^{\shortmid }=\alpha \xi
,~~~\phi ^{\shortmid }=\phi ,\text{~~~}g^{\shortmid }=\frac{1}{\alpha ^{2}}g,%
\text{~~~}\alpha \neq 0,~~~\alpha \in \mathbb{R},
\end{equation*}%
where $\alpha $ is a non-zero real constant. So we get an almost $\alpha $%
-Kenmotsu structure $(\phi ^{\shortmid },\xi ^{\shortmid },\eta ^{\shortmid
},g^{\shortmid }).$ This deformation is called a homothetic deformation (see 
\cite{kimpak}, \cite{olszak}). It is important to note that almost $\alpha $%
-Kenmotsu structures are related to some special local conformal
deformations of almost cosymplectic structures, (see \cite{vaisman}).

If we join these two classes, we obtain a new notion of an almost $\alpha $%
-cosymplectic manifold, which is defined by the following formula%
\begin{equation*}
d\eta =0,\text{ \ \ }d\Phi =2\alpha \eta \wedge \Phi ,
\end{equation*}%
for any real number $\alpha ,$ (see \cite{kimpak}).\ Obviously, a normal
almost $\alpha $-cosymplectic manifold is an $\alpha $-cosymplectic
manifold. An $\alpha $-cosymplectic manifold is either cosymplectic under
the condition $\alpha =0$ or $\alpha $-Kenmotsu ($\alpha \neq 0)$ for $%
\alpha \in \mathbb{R}.$

We denote by $\mathcal{D}$ the distribution orthogonal to $\xi $, that is
\linebreak $\mathcal{D}=\ker (\eta )=\left\{ X:\eta (X)=0\right\} $ and let $%
M^{2n+1}$ be an almost $\alpha $-cosymplectic manifold with structure $(\phi
,\xi ,\eta ,g)$. Since the $1$-form is closed, we have $\mathcal{L}_{\xi
}\eta =0$ and $[X,\xi ]\in \mathcal{D}$ for any $X\in \mathcal{D}.$ The
Levi-Civita connection satisfies $\nabla _{\xi }\xi =0$ and $\nabla _{\xi
}\phi \in \mathcal{D},$ which implies that $\nabla _{\xi }X\in \mathcal{D}$
for any $X\in \mathcal{D}.$

Now, we set $A=-\nabla \xi $ and $h=\frac{1}{2}\mathcal{L}_{\xi }\phi .$
Obviously, $A(\xi )=0$ and $h(\xi )=0.$ Moreover, the tensor fields $A$ and $%
h$ are symmetric operators and satisfy the following relations%
\begin{equation}
\nabla _{X}\xi =-\alpha \phi ^{2}X-\phi hX,  \label{2.3}
\end{equation}%
\begin{equation}
(\phi \circ h)X+(h\circ \phi )X=0,\text{ \ \ }(\phi \circ A)X+(A\circ \phi
)X=-2\alpha \phi ,  \label{2.4}
\end{equation}%
\begin{eqnarray}
(\nabla _{X}\eta )Y &=&\alpha \left[ g(X,Y)-\eta (X)\eta (Y)\right] +g(\phi
Y,hX),  \label{2.5} \\
\delta \eta &=&-2\alpha n,\text{ \ \ }tr(h)=0,  \label{2.0}
\end{eqnarray}%
for any vector fields $X,Y$ on $M^{2n+1}.$ We also remark that 
\begin{equation}
h=0\Leftrightarrow \nabla \xi =-\alpha \phi ^{2}.  \label{2.6}
\end{equation}%
From (\cite{kimpak}, Lemma $2.2$), we have%
\begin{equation}
(\nabla _{X}\phi )Y+(\nabla _{\phi X}\phi )\phi Y=-\alpha \eta (Y)\phi
X-2\alpha g(X,\phi Y)\xi -\eta (Y)hX,  \label{2.7}
\end{equation}%
for any vector fields $X,Y$ on $M^{2n+1}.$

Olszak proved that the integral submanifold of the distribution $\mathcal{D}$
on an almost cosymplectic manifold has Kaehlerian structures if and only if
it satisfies the following condition%
\begin{equation*}
(\nabla _{X}\phi )Y=-g(\phi AX,Y)\xi +\eta (Y)\phi AX,
\end{equation*}%
where $A$ is defined by $A=\phi h.$ Analogously, we give the following
proposition.

\begin{proposition}
Let $M^{2n+1}$ be an almost $\alpha $-cosymplectic manifold. The integral
submanifold of the distribution $\mathcal{D}$ on an almost $\alpha $%
-cosymplectic manifold has Kaehlerian structures if and only if it satisfies
the condition%
\begin{equation}
(\nabla _{X}\phi )Y=-g(\phi AX,Y)\xi +\eta (Y)\phi AX,  \label{7.71}
\end{equation}%
where $A$ is given by $A=\alpha \phi ^{2}+\phi h,$ for any vector fields $%
X,Y $ on $M^{2n+1}$.
\end{proposition}

Let $(M^{2n+1},\phi ,\xi ,\eta ,g)$ be an almost $\alpha $-cosymplectic
manifold. We denote the curvature tensor and Ricci tensor of $g$ by $R$ and $%
S$ respectively. We define a self adjoint operator $l=R(.,\xi )\xi $ (The
Jacobi operator with respect to $\xi $). By simple computations, we have the
following equations

\begin{eqnarray}
R(X,Y)\xi &=&\alpha ^{2}\left[ \eta (X)Y-\eta (Y)X\right] -\alpha \left[
\eta (X)\phi hY-\eta (Y)\phi hX\right]  \label{3.1} \\
&&+(\nabla _{Y}\phi h)X-(\nabla _{X}\phi h)Y,  \notag
\end{eqnarray}

\begin{equation}
lX=\alpha ^{2}\phi ^{2}X+2\alpha \phi hX-h^{2}X+\phi (\nabla _{\xi }h)X,
\label{3.2}
\end{equation}%
\begin{equation}
lX-\phi l\phi X=2\left[ \alpha ^{2}\phi ^{2}X-h^{2}X\right] ,  \label{3.4}
\end{equation}%
\begin{equation}
(\nabla _{\xi }h)X=-\phi lX-\alpha ^{2}\phi X-2\alpha hX-\phi h^{2}X,
\label{3.3}
\end{equation}%
\begin{equation}
S(X,\xi )=-2n\alpha ^{2}\eta (X)-(\func{div}(\phi h))X,  \label{3.5}
\end{equation}%
\begin{equation}
S(\xi ,\xi )=-\left[ 2n\alpha ^{2}+tr(h^{2})\right] ,  \label{3.6}
\end{equation}%
where $X,Y$ arbitrary vector fields on $M^{2n+1}.$

\section{Almost $\protect\alpha $-cosymplectic manifolds with $\protect\eta $%
-parallelism}

Let $(M^{2n+1},\phi ,\xi ,\eta ,g)$ be an almost $\alpha $-cosymplectic
manifold. For any vector field $X$ on $M^{2n+1}$, we can take $X=X^{T}+\eta
(X)\xi ,$ $X^{T}$ is tangentially part of $X$ and $\eta (X)\xi $ the normal
part of $X.$ We say that any symmetric $(1,1)$-type tensor field $B$ on a $%
(M^{2n+1},\phi ,\xi ,\eta ,g)$ is said to be a $\eta $-parallel tensor if it
satisfies the equation 
\begin{equation*}
g(\left( \nabla _{X^{T}}B\right) Y^{T},Z^{T})=0,
\end{equation*}%
for all tangent vectors $X^{T},Y^{T},Z^{T}$ orthogonal to $\xi .$

Dileo and Pastore study almost Kenmotsu manifolds such that $h^{\prime
}=h\circ \phi $ is $\eta -$parallel and prove that this condition is not to
equivalent to the characteristic vector field $\xi $ belongs to the $(\kappa
,\mu )^{\prime }-$nullity distribution (\cite{pastore3}) for some constants $%
\kappa $ and $\mu $, that is the Riemannian curvature satisfies%
\begin{equation*}
R(X,Y)\xi =\kappa (\eta (Y)X-\eta (X)Y)+\mu (\eta (Y)h^{\prime }X-\eta
(X)h^{\prime }Y),
\end{equation*}%
for all vector fields $X$ and $Y$. For this, taking into account the
spectrum of the operator $h^{\prime }$, which is of type $\left\{ 0,\lambda
_{1},-\lambda _{1},...,\lambda _{r},-\lambda _{r}\right\} $, each $\lambda
_{i}$, being positive constant function along $\mathcal{D}$, they prove that
an integral submanifold of $\overset{\sim }{M}$ of $\mathcal{D}$ is locally
the Riemannian product $M_{0}\times M_{\lambda _{1}}\times M_{-\lambda
_{1}}\times ...\times M_{\lambda _{r}}\times M_{-\lambda _{r}}$, where $%
M_{0},M_{\lambda _{r}}$ and $M_{-\lambda _{r}}$ are integral submanifolds of
of the distributions of the eigenvectors with eigenvalues $0,\lambda _{i}$
and $-\lambda _{i}$ respectively. Moreover, $M_{0}$ is is an almost
Kaehlerian manifold and each $M_{\lambda _{i}}\times M_{-\lambda _{i}}$ is a
bi-Lagrangian Kaehlerian manifold and the structure is CR-integrable if and
only if$\ 0$ is a simple eigenvalue or $M_{0}$ is a Kaehlerian manifold.
Also, the authors consider almost Kenmotsu manifolds with $\eta -$parallel
and satisfies the additional condition $\nabla _{\xi }h^{\prime }=0$, the
almost Kenmotsu manifold is locally a warped product (\cite{pastore2}).

In this section, we study almost $\alpha $-cosymplectic manifolds with $\eta 
$-parallel tensor field $h$ under the condition $\nabla _{\xi }h=0$ and the
relation between $\eta $-parallellity of the tensor field $h$ and the
distribution $\mathcal{D}$. In their works Pastore and Dileo studied $\eta $%
-parallellity of the tensor field $h\phi $ in almost Kenmotsu manifolds $%
(\alpha =1)$ which is a particular case of almost $\alpha $-cosymplectic
manifolds (see (\cite{pastore2})). In this work, we will be especially
focused on providing conditions under which $h$ is $\eta $-parallel in
almost $\alpha $-cosymplectic structures. For such manifolds, we also give
results about some certain tensor conditions.

The starting point of the investigation of almost $\alpha $-cosymplectic
manifolds with $\eta $-parallel tensor $h$ is the following propositions:

\begin{proposition}
Let $(M^{2n+1},\phi ,\xi ,\eta ,g)$ be an almost $\alpha $-cosymplectic
manifold. If the tensor field $h$ is $\eta $-parallel, then we have%
\begin{eqnarray}
(\nabla _{X}h)Y &=&-\eta (X)\left[ \phi lY+\alpha ^{2}\phi Y+2\alpha hY+\phi
h^{2}Y\right]  \label{4.1} \\
&&-\eta (Y)\left[ -\alpha \phi ^{2}hX+\phi h^{2}X\right] +g(Y,\alpha hX+\phi
h^{2}X)\xi ,  \notag
\end{eqnarray}%
for all vector fields $X,Y$ on $M^{2n+1},$ where the tensor $l=R(.,\xi )\xi $
is the Jacobi operator with respect to the characteristic vector field $\xi $
and $h$ is a $(1,1)$-type tensor field.
\end{proposition}

\begin{proof}
We suppose that $h$ is $\eta $-parallel. If we denote by $X^{T}$ the
component of $X$ orthogonal to $\xi $, then we get%
\begin{equation*}
\begin{array}{llll}
0 & = & g(\left( \nabla _{X^{T}}h\right) Y^{T},Z^{T})=g(\left( \nabla
_{X-\eta (X)\xi }h\right) (Y-\eta (Y)\xi ),Z-\eta (Z)\xi ) &  \\ 
& = & g(\left( \nabla _{X}h\right) Y,Z)-\eta (X)g(\left( \nabla _{\xi
}h\right) Y,Z)-\eta (Y)g(\left( \nabla _{X}h\right) \xi ,Z) &  \\ 
&  & -\eta (Z)g(\left( \nabla _{X}h\right) Y,\xi )+\eta (X)\eta (Y)g(\left(
\nabla _{\xi }h\right) \xi ,Z)+\eta (Y)\eta (Z)g(\left( \nabla _{X}h\right)
\xi ,\xi ) &  \\ 
&  & +\eta (Z)\eta (X)g(\left( \nabla _{\xi }h\right) Y,\xi )-\eta (X)\eta
(Y)\eta (Z)g(\left( \nabla _{\xi }h\right) \xi ,\xi ), & 
\end{array}%
\end{equation*}%
for all vector fields $X,Y,Z$ on $M^{2n+1}.$ If we simplify the above
equation, we get%
\begin{equation*}
\begin{array}{lll}
0 & = & g(\left( \nabla _{X}h\right) Y,-\phi ^{2}Z)-\eta (X)g(\left( \nabla
_{\xi }h\right) Y,Z)-\eta (Y)g(\left( \nabla _{X}h\right) \xi ,Z).%
\end{array}%
\end{equation*}%
Using (\ref{2.3}), (\ref{2.4}) and (\ref{3.3}), we obtain (\ref{4.1}).
\end{proof}

\begin{proposition}
Let $(M^{2n+1},\phi ,\xi ,\eta ,g)$ be an almost $\alpha $-cosymplectic
manifold such that $h$ is $\eta $-parallel. If in addition $\nabla _{\xi
}h=0 $, then the eigenvalues of $h$ are constant on $M^{2n+1}.$
\end{proposition}

\begin{proof}
Let $\lambda $ be an eigen function of $h$ and $Y$ be a local unit vector
field orthogonal to $\xi $ such that $h(Y)=\lambda Y.$ Since $h$ is $\eta $%
-parallel, using (\ref{4.1}) we have 
\begin{equation}
g(\left( \nabla _{X}h\right) Y,Y)=\eta (X)\xi (\lambda ),  \label{4.2}
\end{equation}%
for any vector field $X.$ Also, the left-hand side of the equation (\ref{4.3}%
) can be written as%
\begin{equation}
g(\left( \nabla _{X}h\right) Y,Y)=X(\lambda ),  \label{4.3}
\end{equation}%
for any $Y\in \mathcal{D}$. From (\ref{4.2}) and (\ref{4.3}), we get $%
d\lambda =\xi (\lambda )\otimes \eta ,$ for $X\in \chi (M^{2n+1}).$ On the
other hand, if $\nabla _{\xi }h=0,$ then $\xi (\lambda )=0$ for any eigen
function $\lambda .$ Thus we obtain $d\lambda =0$ which completes the proof.
\end{proof}

\begin{theorem}
Let $(M^{2n+1},\phi ,\xi ,\eta ,g)$ be an almost $\alpha $-cosymplectic
manifold such that $h$ is $\eta $-parallel and $\nabla _{\xi }h=0$. If the
sectional curvatures of all plane sections $\xi $ are $\neq \alpha ^{2}$ at
some point, then $h$ has eigenvalues $\lambda _{i}\neq 0$ on $\mathcal{D}$.
\end{theorem}

\begin{proof}
Assuming the maximal open subset $W$ of $M^{2n+1}$ such that the
multiplicities of the eigenvalue functions of $h$ are constant on each
connected component of $W.$ Let $W^{\ast }$ be such a connected component of 
$W$ and $\lambda _{j}$ $(j=1,\ldots ,m)$ the distinct eigenvalue functions
of $h$ restricted to $\mathcal{D}$ on $W^{\ast }.$ If $X_{j}$ is a local
unit vector field of $h$ such that $h(X_{j})=\lambda _{j}X_{j}$, then the
spectrum $D(\lambda _{j})$ can be written%
\begin{equation*}
D(\lambda _{j})=\left\{ X_{j}:\text{ }h(X_{j})=\lambda _{j}X_{j}\right\} ,
\end{equation*}%
for $X_{j}\in \mathcal{D}$. As $h$ anti-commutes with $\phi ,$ it follows
that $h(\phi X_{j})=-\lambda _{j}\phi X_{j}.$ Since $M^{2n+1}$ is connected,
the eigenvalues of $h$ are constant on $M^{2n+1}.$ Now, we suppose that $%
hX_{j}=0$ for some unit vector fields. Using (\ref{4.1}) and $\nabla _{\xi
}h=0,$ we have%
\begin{equation*}
\begin{array}{c}
(\nabla _{\xi }h)X_{j}=-\phi R(X_{j},\xi )\xi -\alpha ^{2}\phi X_{j}-2\alpha
hX_{j}-\phi h^{2}X_{j} \\ 
g(\phi R(X_{j},\xi )\xi ,\phi X_{j})=-\alpha ^{2}.%
\end{array}%
\end{equation*}%
This means $K(\xi ,X_{j})=-\alpha ^{2}$ everywhere on $M^{2n+1}$ that
contradicts our assumption. Thus we obtain $hX_{j}\neq 0$. Hence, $h$ cannot
have an eigenvalue to $0$ on $\mathcal{D}$. Consequently, $h$ is
non-degenerate on the distribution $\mathcal{D}$.
\end{proof}

\begin{remark}
If the sectional curvatures of all plane sections $\xi $ are equal to $%
\alpha ^{2}$ at some point, then we have%
\begin{equation*}
g(R(X,\xi )\xi ,X)+g(R(\phi X,\xi )\xi ,\phi X)=2\left[ \alpha ^{2}g(\phi
^{2}X,X)-g(h^{2}X,X)\right] ,
\end{equation*}%
for any unit vector field $X$ on $\mathcal{D}$. The above equation reduces
to $g(h^{2}X,X)=0$. This equation yields $trace(h^{2})=0$ and implies that $%
h=0.$ So this condition guarantee that $h$ is not equal to $0.$
\end{remark}

\begin{proposition}
Let $(M^{2n+1},\phi ,\xi ,\eta ,g)$ be an almost $\alpha $-cosymplectic
manifold and the characteristic vector space of $h$ is completely formed by
the direct sum $D(\lambda )\oplus D(-\lambda )$ on $\mathcal{D}$. Then the
tensor field $h$ satisfies the following relation%
\begin{equation}
h^{2}=\lambda ^{2}(I-\eta \otimes \xi ),  \label{7.47}
\end{equation}%
for any vector fields.
\end{proposition}

\begin{proof}
If we denote by $X^{T}$ the component of $X$ orthogonal to $\xi ,$ then we
have $h^{2}X^{T}=\lambda ^{2}X^{T}$ for any eigen functions $\lambda $ on $%
\mathcal{D}$, where $X^{T}=X-\eta (X)\xi $ for any vector field $X.$ So we
get $h^{2}X=\lambda ^{2}(X-\eta (X)\xi ).$ Thus this completes the proof.
\end{proof}

\begin{proposition}
Let $(M^{2n+1},\phi ,\xi ,\eta ,g)$ be an almost $\alpha $-cosymplectic
manifold. Then we have%
\begin{eqnarray}
&&g(R_{\xi X}Y,Z)-g(R_{\xi X}\phi Y,\phi Z)+g(R_{\xi \phi X}Y,\phi
Z)+g(R_{\xi \phi X}\phi Y,Z)  \notag \\
&=&2(\nabla _{hX}\Phi )(Y,Z)+2\alpha ^{2}\eta (Y)g(X,Z)-2\alpha ^{2}\eta
(Z)g(X,Y)  \notag \\
&&-2\alpha \eta (Y)g(\phi hX,Z)+2\alpha \eta (Z)g(\phi hX,Y),  \label{113}
\end{eqnarray}%
for any vector fields $X,Y,Z$ on $M^{2n+1}.$
\end{proposition}

\begin{proof}
This formula is proved by Pastore for an almost Kenmotsu manifold (see \cite%
{pastore3}).
\end{proof}

\begin{theorem}
Let $(M^{2n+1},\phi ,\xi ,\eta ,g)$ be an almost $\alpha $-cosymplectic
manifold, $\nabla _{\xi }h=0$ and the sectional curvatures of all plane
sections $\xi $ are $\neq \alpha ^{2}$ at some point. If the tensor field $h$
is $\eta $-parallel, then the integral submanifolds of the distribution $%
\mathcal{D}$ have Kaehlerian structures.
\end{theorem}

\begin{proof}
Let $X^{T},Y^{T},Z^{T}$ be an orthogonal vector fields to $\xi .$ Using Eqs.
(\ref{3.1}) and (\ref{4.1}) we obtain%
\begin{equation}
g(R(Y^{T},Z^{T})\xi ,X^{T})=g((\nabla _{Y^{T}}\phi )Z^{T}-(\nabla
_{Z^{T}}\phi )Y^{T},hX^{T}),  \label{4.4}
\end{equation}%
for any vector fields $X,Y,Z,$ where $(\nabla _{X}\phi h)Y=-(\nabla
_{X}h)\phi Y-h(\nabla _{X}\phi )Y.$ In view of (\ref{4.4}), we obtain%
\begin{equation}
-g(R(\xi ,X^{T})\phi Y^{T},\phi Z^{T})=-g((\nabla _{\phi Y^{T}}\phi )\phi
Z^{T}-(\nabla _{\phi Z^{T}}\phi )\phi Y^{T},hX^{T}),  \label{4.5}
\end{equation}%
\begin{equation}
g(R(\xi ,\phi X^{T})Y^{T},\phi Z^{T})=g((\nabla _{Y^{T}}\phi )\phi
Z^{T}-(\nabla _{\phi Z^{T}}\phi )Y^{T},h\phi X^{T}),  \label{4.6}
\end{equation}%
\begin{equation}
g(R(\xi ,\phi X^{T})\phi Y^{T},Z^{T})=g((\nabla _{\phi Y^{T}}\phi
)Z^{T}-(\nabla _{\phi Z^{T}}\phi )\phi Y^{T},h\phi X^{T}).  \label{4.7}
\end{equation}%
Consider Eqs. (\ref{4.5}), (\ref{4.6}), (\ref{4.7}) and taking sum on both
sides of those equalities, respectively, we find%
\begin{equation}
\begin{array}{l}
g(R(\xi ,X^{T})Y^{T},Z^{T})-g(R(\xi ,X^{T})\phi Y^{T},\phi Z^{T})+g(R(\xi
,\phi X^{T})Y^{T},\phi Z^{T}) \\ 
+g(R(\xi ,\phi X^{T})\phi Y^{T},Z^{T})=g((\nabla _{Y^{T}}\phi
)Z^{T},hX^{T})-g((\nabla _{Z^{T}}\phi )Y^{T},hX^{T}) \\ 
-g((\nabla _{\phi Y^{T}}\phi )\phi Z^{T},hX^{T})+g((\nabla _{\phi Z^{T}}\phi
)\phi Y^{T},hX^{T})+g((\nabla _{Y^{T}}\phi )\phi Z^{T},h\phi X^{T}) \\ 
-g((\nabla _{\phi Z^{T}}\phi )Y^{T},h\phi X^{T})+g((\nabla _{\phi Y^{T}}\phi
)Z^{T},h\phi X^{T})-g((\nabla _{\phi Z^{T}}\phi )\phi Y^{T},h\phi X^{T}).%
\end{array}
\label{4.8}
\end{equation}%
We will also need the equality%
\begin{equation*}
\nabla _{X^{T}}\phi ^{2}Y^{T}=-\nabla _{X^{T}}Y^{T},
\end{equation*}%
which holds for vector fields $X,Y,Z$ orthogonal to $\xi .$ From the above
equation, we have%
\begin{equation}
(\nabla _{X^{T}}\phi )\phi Y^{T}=\nabla _{X^{T}}\phi ^{2}Y^{T}-\phi \nabla
_{X^{T}}\phi Y^{T}.  \label{4.9}
\end{equation}%
Taking inner product of the equation (\ref{4.9}) with $Z^{T}$ we get%
\begin{equation}
g((\nabla _{X^{T}}\phi )\phi Y^{T},Z^{T})=g((\nabla _{X^{T}}\phi )Y^{T},\phi
Z^{T}).  \label{4.10}
\end{equation}%
According to Eq. (\ref{4.10}) we also have%
\begin{equation*}
g((\nabla _{X^{T}}\phi )Y^{T},Z^{T})=-g((\nabla _{\phi X^{T}}\phi )\phi
Y^{T},Z^{T}).
\end{equation*}%
If we substitute (\ref{4.9}) and (\ref{4.10}) into (\ref{4.8}), then we
obtain%
\begin{eqnarray*}
&&g(R(\xi ,X^{T})Y^{T},Z^{T})-g(R(\xi ,X^{T})\phi Y^{T},\phi Z^{T}) \\
&&+g(R(\xi ,\phi X^{T})Y^{T},\phi Z^{T})+g(R(\xi ,\phi X^{T})\phi
Y^{T},Z^{T}) \\
&=&g((\nabla _{Y^{T}}\phi )Z^{T},hX^{T})-g((\nabla _{Z^{T}}\phi
)Y^{T},hX^{T})-g((\nabla _{\phi Y^{T}}\phi )\phi Z^{T},hX^{T})
\end{eqnarray*}%
\begin{eqnarray}
&&+g((\nabla _{\phi Z^{T}}\phi )\phi Y^{T},hX^{T})+g((\nabla _{Y^{T}}\phi
)\phi Z^{T},h\phi X^{T})-g((\nabla _{\phi Z^{T}}\phi )Y^{T},h\phi X^{T}) 
\notag \\
&&+g((\nabla _{\phi Y^{T}}\phi )Z^{T},h\phi X^{T})-g((\nabla _{\phi
Z^{T}}\phi )\phi Y^{T},h\phi X^{T})  \notag \\
&=&4[g((\nabla _{Y^{T}}\phi )Z^{T},hX^{T})-g((\nabla _{Z^{T}}\phi
)Y^{T},hX^{T})].  \label{4.11}
\end{eqnarray}%
The left-hand side of Eq. (\ref{4.11}) using by Eq. (\ref{113}) we can be
written as follows:%
\begin{equation*}
\begin{array}{c}
g(R(\xi ,X^{T})Y^{T},Z^{T})-g(R(\xi ,X^{T})\phi Y^{T},\phi Z^{T})+g(R(\xi
,\phi X^{T})Y^{T},\phi Z^{T}) \\ 
+g(R(\xi ,\phi X^{T})\phi Y^{T},Z^{T})=2(\nabla _{hX}\Phi
)(Y,Z)=2g(Y^{T},(\nabla _{hX^{T}}\phi )Z^{T}).%
\end{array}%
\end{equation*}%
Combining this equality, we get%
\begin{equation*}
g(Y^{T},(\nabla _{hX^{T}}\phi )Z^{T})=2[g((\nabla _{Y^{T}}\phi
)Z^{T},hX^{T})-g((\nabla _{Z^{T}}\phi )Y^{T},hX^{T})].
\end{equation*}%
In the above equation, if we restrict $X$ to $\mathcal{D}$ and replace by $%
h^{-1}X$, then we have 
\begin{equation*}
g(Y^{T},(\nabla _{X^{T}}\phi )Z^{T})=2[g((\nabla _{Y^{T}}\phi
)Z^{T},X^{T})-g((\nabla _{Z^{T}}\phi )Y^{T},X^{T})].
\end{equation*}%
Since the tensor field $h$ is non-degenerate, its invertible on $\mathcal{D}$%
. Thus we obtain%
\begin{eqnarray}
&&g((\nabla _{X^{T}}\phi )Y^{T},Z^{T})+g((\nabla _{Y^{T}}\phi
)Z^{T},X^{T})+g((\nabla _{Z^{T}}\phi )X^{T},Y^{T})  \notag \\
&=&-2[g((\nabla _{Y^{T}}\phi )Z^{T},X^{T})-g((\nabla _{Z^{T}}\phi
)Y^{T},X^{T})  \notag \\
&&+g((\nabla _{Z^{T}}\phi )X^{T},Y^{T})-g((\nabla _{X^{T}}\phi )Z^{T},Y^{T})
\notag \\
&&+g((\nabla _{Z^{T}}\phi )Y^{T},X^{T})-g((\nabla _{Y^{T}}\phi )Z^{T},X^{T})]
\notag \\
&=&2[g((\nabla _{Z^{T}}\phi )X^{T},Y^{T})-g((\nabla _{X^{T}}\phi
)Z^{T},Y^{T})]  \label{4.12}
\end{eqnarray}%
The left-hand side of Eq. (\ref{4.12}) is equal to $d\Phi
(X^{T},Y^{T},Z^{T}),$ where $\Phi $ is the two-form is given by $\Phi
(X,Y)=g(X,\phi Y).$ On the other hand, since $M^{2n+1}$ is an almost $\alpha 
$-cosymplectic manifold, then we have%
\begin{equation}
d\Phi (X^{T},Y^{T},Z^{T})=2\alpha (\eta (X^{T})\Phi (Y^{T},Z^{T})+\eta
(Z^{T})\Phi (X^{T},Y^{T})+\eta (Y^{T})\Phi (Z^{T},X^{T}))=0  \label{4.13}
\end{equation}%
Thus Eq. (\ref{4.13}) reduces to%
\begin{equation}
g((\nabla _{Z^{T}}\phi )X^{T},Y^{T})=g((\nabla _{X^{T}}\phi )Z^{T},Y^{T}).
\label{4.14}
\end{equation}%
Moreover, we get%
\begin{equation*}
g((\nabla _{Y^{T}}\phi )Z^{T},X^{T})=g((\nabla _{Z^{T}}\phi
)Y^{T},X^{T})=-g((\nabla _{Y^{T}}\phi )Z^{T},X^{T}),
\end{equation*}%
by using (\ref{4.14}) and anti-commute property of $\phi .$ Hence, we obtain 
\begin{equation}
g((\nabla _{Y^{T}}\phi )Z^{T},X^{T})=0.  \label{4.15}
\end{equation}%
As this is valid for all vector fields $X^{T},Y^{T},Z^{T}$ orthogonal to $%
\xi ,$ this is equivalent to the equality%
\begin{eqnarray}
(\nabla _{Y}\phi )Z &=&[g(Z,hY)-\alpha g(\phi Z,Y)]\xi -\eta (Z)[\alpha \phi
Y+hY]  \notag \\
&=&-g(\phi AY,Z)\xi +\eta (Z)\phi AY,  \label{4.16}
\end{eqnarray}%
for all vector fields $X,Y,Z$ on $M^{2n+1}.$ Therefore, Eq. (\ref{4.16})
implies that the integral submanifolds of the distribution $\mathcal{D}$ are
Kaehlerian.
\end{proof}

\section{Almost $\protect\alpha $-cosymplectic $(\protect\kappa ,\protect\mu %
,\protect\nu )$-spaces}

\subsection{$D$-homothetic deformations}

Let $M^{2n+1}$ be an almost $\alpha $-cosymplectic manifold and $(\phi ,\xi
,\eta ,g)$ its almost $\alpha $-cosymplectic structure. Let $\mathcal{R}%
_{\eta }(M^{2n+1})$ be the subring of the ring of smooth functions $f$ on $%
M^{2n+1}$ such that $df\wedge \eta =0.$

Consider a $D$-homothetic deformation of $(\phi ,\xi ,\eta ,g)$ into an
almost contact metric structure $(\phi ^{^{\prime }},\xi ^{^{\prime }},\eta
^{^{\prime }},g^{^{\prime }})$ defined as%
\begin{equation}
\phi ^{\prime }=\phi ,\text{ }\xi ^{\prime }=\frac{1}{\beta }\xi ,\text{ }%
\eta ^{\prime }=\beta \eta ,\text{ }g^{\prime }=\gamma g+(\beta ^{2}-\gamma
)\eta \otimes \eta ,  \label{504}
\end{equation}%
where $\gamma $ is positive constant and $\beta \in \mathcal{R}_{\eta
}(M^{2n+1})$, $\beta \neq 0$ at any point of $M^{2n+1}.$ Since $d\beta
\wedge \eta =0,$ it follows that%
\begin{equation*}
d\eta ^{\prime }=d\beta \wedge \eta +\beta d\eta =0,
\end{equation*}%
and moreover $d\Phi ^{\prime }=2(\frac{\alpha }{\beta })\eta ^{\prime
}\wedge \Phi ^{\prime }$ , since the fundamental two forms $\Phi ,\Phi
^{\prime }$ of the structures are related by $\Phi ^{\prime }=\gamma \Phi .$
Taking $\frac{\alpha }{\beta }=\beta ^{\prime }$, deformed structure $(\phi
^{^{\prime }},\xi ^{^{\prime }},\eta ^{^{\prime }},g^{^{\prime }})$ can be
written 
\begin{equation*}
\Phi ^{\prime }=\gamma \Phi ,\text{ }d\eta ^{\prime }=0,\text{ }d\Phi
^{\prime }=2\beta ^{\prime }\eta ^{\prime }\wedge \Phi ^{\prime },
\end{equation*}%
for $d\beta =d\beta (\xi )\eta $ and $\beta ^{\prime }=\frac{\alpha }{\beta }%
\in \mathcal{R}_{\eta }(M^{2n+1}).$

Thus a $D$-homothetic deformation of an almost $\alpha $-cosymplectic
structure $(\phi ,\xi ,\eta ,g)$ gives a new almost $(\frac{\alpha }{\beta }%
) $-cosymplectic structure $(\phi ^{^{\prime }},\xi ^{^{\prime }},\eta
^{^{\prime }},g^{^{\prime }})$ on the same manifold.

\begin{proposition}
Let $(M^{2n+1},\phi ,\xi ,\eta ,g)$ be an almost $\alpha $-cosymplectic
manifolds. For $D$-homothetic deformations of almost $\alpha $-cosymplectic
structures on $M^{2n+1}$, the Levi-Civita connections $\nabla ^{\prime }$
and $\nabla $ are related by%
\begin{equation}
\nabla _{X}^{\prime }Y=\nabla _{X}Y-\left( \frac{\beta ^{2}-\gamma }{\beta
^{2}}\right) g(AX,Y)\xi +\frac{d\beta (\xi )}{\beta }\eta (X)\eta (Y)\xi .
\label{7.20}
\end{equation}
\end{proposition}

\begin{proof}
Using Kozsul's formula we have%
\begin{eqnarray*}
2g^{\prime }(\nabla _{X}^{\prime }Y,Z) &=&Xg^{\prime }(Y,Z)+Yg^{\prime
}(X,Z)-Zg^{\prime }(X,Y) \\
&&+g^{\prime }(\left[ X,Y\right] ,Z)+g^{\prime }(\left[ Z,X\right]
,Y)+g^{\prime }(\left[ Z,Y\right] ,X),
\end{eqnarray*}%
for any vector fields $X,Y,Z.$ By applying $g^{\prime }=\gamma g+(\beta
^{2}-\gamma )\eta \otimes \eta $ with all components of Kozsul's formula,
then we find%
\begin{eqnarray*}
2g^{\prime }(\nabla _{X}^{\prime }Y,Z) &=&2\gamma g(\nabla _{X}Y,Z)+2\beta
d\beta (\xi )\eta (X)\eta (Y)\eta (Z) \\
&&+(\beta ^{2}-\gamma )\left[ 2\eta (\nabla _{X}Y)\eta (Z)+2g(Y,\nabla
_{X}\xi )\eta (Z)\right] .
\end{eqnarray*}%
Also, since we have 
\begin{equation*}
2g^{\prime }(\nabla _{X}^{\prime }Y,Z)=2\gamma g(\nabla _{X}^{\prime
}Y,Z))+2(\beta ^{2}-\gamma )\eta (\nabla _{X}^{\prime }Y)\eta (Z),
\end{equation*}%
we obtain the formula%
\begin{equation}
\gamma g(\nabla _{X}^{\prime }Y,Z))+(\beta ^{2}-\gamma )\eta (\nabla
_{X}^{\prime }Y)\eta (Z)=\gamma g(\nabla _{X}Y,Z)  \label{7.16}
\end{equation}%
\begin{equation*}
+\beta d\beta (\xi )\eta (X)\eta (Y)\eta (Z)+(\beta ^{2}-\gamma )\eta
(\nabla _{X}Y)\eta (Z)+g(Y,\nabla _{X}\xi )\eta (Z),
\end{equation*}%
where 
\begin{equation}
\eta (\nabla _{X}^{\prime }Y)=\frac{1}{\beta }d\beta (\xi )\eta (X)\eta
(Y)+\eta (\nabla _{X}Y)+\left( \frac{\beta ^{2}-\gamma }{\beta }\right)
g(Y,\nabla _{X}\xi ).  \label{7.17}
\end{equation}%
By using Eq. (\ref{7.17}) into (\ref{7.16}) and making use some
computations, we get Eq. (\ref{7.20}) which completes the proof.
\end{proof}

\begin{proposition}
For $D$-homothetic deformations of almost $\alpha $-cosymplectic structures,
then the following relations are held:%
\begin{equation}
A^{\prime }X=\frac{1}{\beta }AX,\text{ }h^{\prime }X=\frac{1}{\beta }hX,
\label{7.18}
\end{equation}%
\begin{equation}
R^{\prime }(X,Y)\xi ^{\prime }=\frac{1}{\beta }R(X,Y)\xi +\frac{1}{\beta ^{2}%
}d\beta (\xi )\left[ \eta (X)AY-\eta (Y)AX\right] ,  \label{7.19}
\end{equation}%
for any vector fields $X,Y,Z.$
\end{proposition}

\begin{proof}
By using (\ref{2.3}), (\ref{2.4}), (\ref{504}) and (\ref{7.20}), we obtain%
\begin{equation*}
A^{\prime }X=\frac{X(\beta )}{\beta ^{2}}\xi -\frac{1}{\beta }\nabla _{X}\xi
-\frac{1}{\beta ^{2}}d\beta (\xi )\eta (X)\xi .
\end{equation*}%
By considering the above equation and Eqs. (\ref{504}), Eq. (\ref{7.21}),
then we also have the second equality of (\ref{7.18}). In order to prove Eq.
(\ref{7.19}), we may also use the Riemannian curvature tensor and Eqs. (\ref%
{504}), that is, it holds%
\begin{eqnarray}
R^{\prime }(X,Y)\xi ^{\prime } &=&\nabla _{X}^{\prime }\nabla _{Y}^{\prime
}\xi ^{\prime }-\nabla _{Y}^{\prime }\nabla _{X}^{\prime }\xi ^{\prime
}-\nabla _{\left[ X,Y\right] }^{\prime }\xi ^{\prime }  \label{7.22} \\
&=&-\frac{X(\beta )}{\beta ^{2}}\nabla _{Y}\xi -\frac{Y(\beta )}{\beta ^{2}}%
\nabla _{X}\xi +\frac{1}{\beta }\nabla _{X}^{\prime }\nabla _{Y}\xi  \notag
\\
&&-\frac{1}{\beta }\nabla _{Y}^{\prime }\nabla _{X}\xi -\frac{1}{\beta }%
\nabla _{\left[ X,Y\right] }\xi .  \notag
\end{eqnarray}%
In Eq. (\ref{7.22}) making use some computations by using the formula%
\begin{equation*}
\nabla _{X}^{\prime }\nabla _{Y}\xi =\nabla _{X}\nabla _{Y}\xi +\left( \frac{%
\beta ^{2}-\gamma }{\beta ^{2}}\right) g(Y,A^{2}X)\xi ,
\end{equation*}%
we find%
\begin{equation*}
R^{\prime }(X,Y)\xi ^{\prime }=\frac{X(\beta )}{\beta ^{2}}AY-\frac{Y(\beta )%
}{\beta ^{2}}AX+\frac{1}{\beta }R(X,Y)\xi ,
\end{equation*}%
which is a consequence of Eqs. (\ref{7.20}) and (\ref{7.18}).
\end{proof}

\subsection{$(\protect\kappa ,\protect\mu ,\protect\nu )$-spaces}

In this part, we are especially interested in almost almost $\alpha $%
-cosymplectic manifolds whose almost $\alpha $-cosymplectic structure $(\phi
,\xi ,\eta ,g)$ satisfies the condition (\ref{501}) for $\kappa ,\mu ,\nu
\in \mathcal{R}_{\eta }(M^{2n+1}).$ Such manifolds are said to be almost $%
\alpha $-cosymplectic $(\kappa ,\mu ,\nu )$-spaces and $(\phi ,\xi ,\eta ,g)$
be called almost $\alpha $-cosymplectic $(\kappa ,\mu ,\nu )$ -structure. We
will explain why the functions $\kappa ,\mu ,\nu $ are element of $\mathcal{R%
}_{\eta }(M^{2n+1})$ in the latter.

\begin{proposition}
If $(\phi ,\xi ,\eta ,g)$ is an almost $\alpha $-cosymplectic $(\kappa ,\mu
,\nu )$-structure for $D$-homothetic deformations of almost $\alpha $%
-cosymplectic structures, then $(\phi ^{^{\prime }},\xi ^{^{\prime }},\eta
^{^{\prime }},g^{^{\prime }})$ is an almost $(\frac{\alpha }{\beta })$%
-cosymplectic structure with $\kappa ^{\prime },\mu ^{\prime },\nu ^{\prime
}\in \mathcal{R}_{\eta ^{\prime }}(M^{2n+1})$ being related to $\kappa ,\mu
,\nu $ by the following equalities%
\begin{equation}
\kappa ^{\prime }=\frac{\kappa }{\beta ^{2}},\text{ }\mu ^{\prime }=\frac{%
\mu }{\beta },\text{ }\nu ^{\prime }=\frac{\beta \nu -d\beta (\xi )}{\beta
^{2}},  \label{7.23}
\end{equation}%
which holds%
\begin{eqnarray}
R^{\prime }(X,Y)\xi ^{\prime } &=&\beta \kappa ^{\prime }\left[ \eta
(Y)X-\eta (X)Y\right] +\mu ^{\prime }\left[ \eta (Y)hX-\eta (X)hY\right] 
\notag \\
&&+\nu ^{\prime }\left[ \eta (Y)\phi hX-\eta (X)\phi hY\right] .
\label{7.24}
\end{eqnarray}
\end{proposition}

\begin{proof}
Applying Eqs. (\ref{501}), (\ref{504}) and (\ref{7.18}) into (\ref{7.19}),
we get (\ref{7.24}). By using simple computations, we also obtain%
\begin{equation*}
\begin{array}[t]{l}
\left[ \eta (Y)X-\eta (X)Y\right] (\beta \kappa ^{\prime })+\left[ \eta
(Y)hX-\eta (X)hY\right] (\mu ^{\prime }) \\ 
+\left[ \eta (Y)\phi hX-\eta (X)\phi hY\right] (\nu ^{\prime })=\left[ \eta
(Y)X-\eta (X)Y\right] (\dfrac{\kappa }{\beta }) \\ 
\text{ \ \ \ \ \ \ \ \ \ \ \ \ \ \ \ \ \ \ \ \ \ \ \ \ \ \ \ \ \ \ \ \ \ \ \
\ \ \ \ \ \ \ \ }+\left[ \eta (Y)hX-\eta (X)hY\right] (\frac{\mu }{\beta })
\\ 
\text{ \ \ \ \ \ \ \ \ \ \ \ \ \ \ \ \ \ \ \ \ \ \ \ \ \ \ \ \ \ \ \ \ \ \ \
\ \ \ \ \ \ \ \ }+\left[ \eta (Y)\phi hX-\eta (X)\phi hY\right] (\frac{\nu }{%
\beta }-\frac{d\beta (\xi )}{\beta ^{2}}),%
\end{array}%
\end{equation*}%
completing the proof.
\end{proof}

\begin{theorem}
An almost $\alpha $-cosymplectic $(\kappa ,\mu ,\nu )$-structure, $\kappa
<-\alpha ^{2},$ can be $D$-homothetically transformed to an almost $(\frac{%
\alpha }{\beta })$-cosymplectic $(-1-\frac{3\alpha ^{2}+\alpha \nu }{\beta
^{2}},\frac{\mu }{\beta },\frac{2\alpha }{\beta })$-structure with $\beta
^{2}=-(\kappa +\alpha ^{2}).$
\end{theorem}

\begin{proof}
We suppose that $(\phi ,\xi ,\eta ,g)$ be an almost $\alpha $-cosymplectic $%
(\kappa ,\mu ,\nu )$-structure. By making use of $D$-homothetic deformation
of the structure $(\phi ,\xi ,\eta ,g)$ with $\kappa <-\alpha ^{2}$ and $%
\beta ^{2}=-(\kappa +\alpha ^{2}),$ then we obtain an almost $(\frac{\alpha 
}{\beta })$-cosymplectic $\left( \kappa ^{\prime },\mu ^{\prime },\nu
^{\prime }\right) $-structure $(\phi ^{^{\prime }},\xi ^{^{\prime }},\eta
^{^{\prime }},g^{^{\prime }})$ with 
\begin{equation*}
\kappa ^{\prime }=\frac{\kappa }{\beta ^{2}}+\frac{\xi (\beta )}{\beta ^{3}},%
\text{ }\xi (\beta )=-\frac{\xi (\kappa )}{2\beta },\text{ }\xi (\kappa
)=2(\nu -2\alpha )(\kappa +\alpha ^{2}),
\end{equation*}%
by the means of the above proposition, where%
\begin{equation*}
\kappa ^{\prime }=\frac{\kappa -2\alpha ^{2}+\alpha \nu }{\beta ^{2}},\text{ 
}\mu ^{\prime }=\frac{\mu }{\sqrt{-(\kappa +\alpha ^{2})}}.
\end{equation*}%
To prove the theorem, we need the formula $\nu ^{\prime }=\dfrac{2\alpha }{%
\beta }$ by using the equation $\nu ^{\prime }=\frac{\beta \nu -d\beta (\xi )%
}{\beta ^{2}}.$ Thus $(\frac{\kappa -2\alpha ^{2}+\alpha \nu }{\beta ^{2}},%
\frac{\mu }{\beta },\frac{2\alpha }{\beta })$-structure is obtained for the
structure $(\phi ^{^{\prime }},\xi ^{^{\prime }},\eta ^{^{\prime
}},g^{^{\prime }})$ with $\beta ^{2}=-(\kappa +\alpha ^{2}).$
\end{proof}

\begin{remark}
The above theorem is proved by Olszak and Dacko for the case $\alpha =0$
with $\mu ^{\prime }=\frac{\mu }{\sqrt{-\kappa }}$ (see \cite{olszak3}).
\end{remark}

\begin{proposition}
The following relations are held on every $(2n+1)$-dimensional almost $%
\alpha $-cosymplectic $(\kappa ,\mu ,\nu )$-space $(M^{2n+1},\phi ,\xi ,\eta
,g).$%
\begin{equation}
l=-\kappa \phi ^{2}+\mu h+\nu \phi h,  \label{7.28}
\end{equation}%
\begin{equation}
l\phi -\phi l=2\mu h\phi +2\nu h,  \label{7.29}
\end{equation}%
\begin{equation}
h^{2}=(\kappa +\alpha ^{2})\phi ^{2},\text{ for }\kappa \leq -\alpha ^{2},
\label{7.30}
\end{equation}%
\begin{equation}
(\nabla _{\xi }h)=-\mu \phi h+(\nu -2\alpha )h,  \label{7.31}
\end{equation}%
\begin{equation}
\nabla _{\xi }h^{2}=2(\nu -2\alpha )(\kappa +\alpha ^{2})\phi ^{2},
\label{7.32}
\end{equation}%
\begin{equation}
\xi (\kappa )=2(\nu -2\alpha )(\kappa +\alpha ^{2}),  \label{7.33}
\end{equation}%
\begin{eqnarray}
R(\xi ,X)Y &=&\kappa (g(Y,X)\xi -\eta (Y)X)+\mu (g(hY,X)\xi -\eta (Y)hX) 
\notag \\
&&+\nu (g(\phi hY,X)\xi -\eta (Y)\phi hX),  \label{7.34}
\end{eqnarray}%
\begin{equation}
Q\xi =2n\kappa \xi ,  \label{7.35}
\end{equation}%
\begin{equation}
(\nabla _{X}\phi )Y=g(\alpha \phi X+hX,Y)\xi -\eta (Y)(\alpha \phi X+hX),
\label{7.36}
\end{equation}%
\begin{eqnarray}
(\nabla _{X}\phi h)Y-(\nabla _{Y}\phi h)X &=&-(\kappa +\alpha ^{2})(\eta
(Y)X-\eta (X)Y)-\mu (\eta (Y)hX-\eta (X)hY)  \notag \\
&&+(\alpha -\nu )(\eta (Y)\phi hX-\eta (X)\phi hY),  \label{7.25}
\end{eqnarray}%
\begin{eqnarray}
(\nabla _{X}h)Y-(\nabla _{Y}h)X &=&(\kappa +\alpha ^{2})(\eta (Y)\phi X-\eta
(X)\phi Y+2g(\phi X,Y)\xi )  \notag \\
&&+\mu (\eta (Y)\phi hX-\eta (X)\phi hY)  \notag \\
&&+(\alpha -\nu )(\eta (Y)hX-\eta (X)hY),  \label{7.26}
\end{eqnarray}%
for all vector fields $X,Y$ on $M^{2n+1}.$
\end{proposition}

\begin{proof}
From Eq. (\ref{501}) we get%
\begin{equation}
lX=R(X,\xi )\xi =\kappa (X-\eta (X)\xi )+\mu hX+\nu \phi hX.  \label{7.39}
\end{equation}%
Replacing $X$ by $\phi X$ in Eq. (\ref{7.39}), it gives%
\begin{equation*}
l\phi X=\kappa \phi X+\mu \phi hX+\nu \phi ^{2}hX.
\end{equation*}%
Also, applying the tensor field $\phi $ both sides of the last equation we
have%
\begin{equation*}
\phi lX=-\phi \kappa \phi ^{2}X+\phi \mu hX+\phi \nu \phi hX.
\end{equation*}%
Then subtracting the last two equations, we obtain%
\begin{equation*}
l\phi X-\phi lX=\mu (h\phi X-\phi hX)-2\nu \phi ^{2}hX,
\end{equation*}%
completing the proof of Eq. (\ref{7.29}). By using Eq. (\ref{7.39}) we deduce%
\begin{equation}
\phi l\phi X=\phi \kappa \phi X+\phi \mu \phi hX+\phi \nu \phi ^{2}hX.
\label{7.40}
\end{equation}%
Eqs. (\ref{7.39}) and (\ref{7.40}) shows that%
\begin{equation*}
lX-\phi l\phi X=-2\kappa \phi ^{2}X.
\end{equation*}%
Using Eq. (\ref{7.39}) we have 
\begin{equation*}
-2\kappa \phi ^{2}X=2(\alpha ^{2}\phi ^{2}X-h^{2}X),
\end{equation*}%
which gives Eq. (\ref{7.30}). Moreover, differentiating Eq. (\ref{7.30})
along $\xi $ we get%
\begin{eqnarray*}
(\nabla _{\xi }h)X &=&-\phi lX-\alpha ^{2}\phi X-2\alpha hX-\phi h^{2}X, \\
&=&-\kappa \phi X-\mu \phi hX+\nu hX-\alpha ^{2}\phi X-2\alpha hX \\
&&+(\kappa +\alpha ^{2})\phi X.
\end{eqnarray*}%
Alternately, using (\ref{7.30}), we obtain%
\begin{equation*}
\nabla _{\xi }h^{2}=(\nabla _{\xi }h)h+h(\nabla _{\xi }h)=2(\nu -2\alpha
)h^{2}X.
\end{equation*}%
The proof of Eq. (\ref{7.32}) is obvious from Eq. (\ref{7.31}). Then
differentiating Eq. (\ref{7.32}) along $\xi $ we find%
\begin{equation*}
2(\nu -2\alpha )(\kappa +\alpha ^{2})\phi ^{2}X=\left[ \xi (\kappa )\right]
\phi ^{2}X.
\end{equation*}%
Since $g(R(\xi ,X)Y,Z)=g(R(Y,Z)\xi ,X),$ we have 
\begin{eqnarray*}
g(R(Y,Z)\xi ,X) &=&\kappa (\eta (Z)g(Y,X)-\eta (Y)g(Z,X))+\mu (\eta
(Z)g(hY,X)-\eta (Y)g(hZ,X)) \\
&&+\nu (\eta (Z)g(\phi hY,X)-\eta (Y)g(\phi hZ,X)),,
\end{eqnarray*}%
by using Eq. (\ref{501}). The last equation completes the proof of Eq. (\ref%
{7.34}). Contracting Eq. (\ref{7.34}) with respect to $X,Y$ and using the
definition of Ricci tensor, we obtain%
\begin{equation*}
S(\xi ,Z)=\overset{2n+1}{\underset{i=1}{\dsum }}g(R(\xi
,E_{i})E_{i},Z)=2n\kappa \eta (Z),
\end{equation*}%
for any vector field $Z.$ Next, it is clear that Eq. (\ref{7.35}) is valid.
In addition, Eq. (\ref{7.35}) implies that%
\begin{eqnarray*}
g(R_{\xi X}Y,Z) &=&\kappa \left[ g(X,Y)\eta (Z)-\eta (Y)g(X,Z)\right] +\mu %
\left[ g(hX,Y)\eta (Z)-\eta (Y)g(hX,Z)\right] \\
&&+\nu \left[ g(\phi hY,X)\eta (Z)-\eta (Y)g(\phi hX,Z)\right] .
\end{eqnarray*}%
Accordingly, combining the last equation and Eq. (\ref{113}), we deduce that%
\begin{equation*}
-2\kappa \left[ \eta (Y)g(X,Z)-\eta (Z)g(X,Y)\right] .
\end{equation*}%
So Eq. (\ref{113}) reduces to the equation%
\begin{eqnarray*}
-2\kappa \left[ \eta (Y)g(X,Z)-\eta (Z)g(X,Y)\right] &=&2(\nabla _{hX}\Phi
)(Y,Z)+2\alpha ^{2}\eta (Y)g(X,Z) \\
&&-2\alpha ^{2}\eta (Z)g(X,Y)-2\alpha \eta (Y)g(\phi hX,Z) \\
&&+2\alpha \eta (Z)g(\phi hX,Y).
\end{eqnarray*}%
In view of the last equation, we obtain%
\begin{eqnarray}
-(\nabla _{hX}\Phi )(Y,Z) &=&(\kappa +\alpha ^{2})\left[ \eta (Y)g(X,Z)-\eta
(Z)g(X,Y)\right]  \label{7.41} \\
&&-\alpha \left[ \eta (Y)g(\phi hX,Z)-\eta (Z)g(\phi hX,Y)\right] .  \notag
\end{eqnarray}%
Substituting $X=hX$ in Eq. (\ref{7.41}) and considering the relation $%
(\nabla _{X}\Phi )(Y,Z)=g((\nabla _{X}\phi )Z,Y),$ then we get%
\begin{eqnarray*}
0 &=&(\kappa +\alpha ^{2})g((\nabla _{X}\phi )Z,Y)-\alpha \left[ \eta
(Y)g(\phi X,Z)-\eta (Z)g(\phi X,Y)\right] \\
&&-\left[ \eta (Y)g(hX,Z)-\eta (Z)g(hX,Y)\right] .
\end{eqnarray*}%
The last equation implies that%
\begin{equation*}
(\nabla _{X}\phi )Z=\alpha \left[ g(\phi X,Z)\xi -\eta (Z)\phi X\right]
+g(hX,Z)\xi -\eta (Z)hX.
\end{equation*}%
Here, replacing $Z$ by $Y$ and organizing the last equation, we obtain Eq. (%
\ref{7.36}). On the other hand, Eq. (\ref{7.36}) can be written as follows:%
\begin{equation*}
(\nabla _{X}\phi )Y=-g(\phi AX,Y)\xi +\eta (Y)\phi AX,
\end{equation*}%
by using the tensor field $A$. Using Eq. (\ref{3.1}), we also have%
\begin{eqnarray}
(\nabla _{X}\phi h)Y-(\nabla _{Y}\phi h)X &=&-R(X,Y)\xi +\alpha ^{2}\left[
\eta (X)Y-\eta (Y)X\right]  \notag \\
&&-\alpha \left[ \eta (X)\phi hY-\eta (Y)\phi hX\right] .  \label{7.42}
\end{eqnarray}%
The proof of Eq. (\ref{7.25}) is obvious from Eq. (\ref{501}). Eq. (\ref%
{7.26}) is an immediate consequence of Eq. (\ref{7.42}). Moreover, Eq. (\ref%
{7.36}) shows that the integral submanifold of the distribution $\mathcal{D}$
is Kaehlerian for an almost $\alpha $-cosymplectic $(\kappa ,\mu ,\nu )$%
-space.
\end{proof}

\begin{remark}
Eq. (\ref{7.36}) shows that almost $\alpha $-cosymplectic $(\kappa ,\mu ,\nu
)$-spaces satisfy the Kaehlerian structure condition.
\end{remark}

Now, we need the following formula for the latter usage.

\begin{proposition}
Let $(M^{2n+1},\phi ,\xi ,\eta ,g)$ be an almost $\alpha $-cosymplectic
manifold with Kaehlerian integral submanifolds. Then the following relation
is valid%
\begin{eqnarray}
Q\phi -\phi Q &=&l\phi -\phi l+4\alpha (1-n)\phi A+4\alpha ^{2}(1-n)\phi X
\label{7.90} \\
&&+(\eta \circ Q\phi )\xi -\eta \circ (\phi Q\xi ),  \notag
\end{eqnarray}%
for all vector fields on $M^{2n+1}.$
\end{proposition}

\begin{proof}
For the curvature transformation of almost $\alpha $-cosymplectic manifold $%
M^{2n+1}$ with Kaehlerian integral submanifolds, we have%
\begin{equation}
\begin{array}{rl}
R(X,Y)\phi Z-\phi R(X,Y)Z= & g(AX,\phi Z)AY-g(AY,\phi Z)AX \\ 
& -g(AX,Z)\phi AY+g(AY,Z)\phi AX \\ 
& +\eta (Z)\phi ((\nabla _{X}A)Y-(\nabla _{Y}A)X) \\ 
& +g(\nabla _{X}A)Y-(\nabla _{Y}A)X,\phi Z)\xi \\ 
= & g(AX,\phi Z)AY-g(AY,\phi Z)AX \\ 
& -g(AX,Z)\phi AY+g(AY,Z)\phi AX \\ 
& -\eta (Z)\phi (R(X,Y)\xi )-g(R(X,Y)\xi ,\phi Z)\xi .%
\end{array}
\label{3a9}
\end{equation}

Then using (\ref{*}) and (\ref{3a9}), one obtains,%
\begin{equation}
\begin{array}{l}
g(\phi R(\phi X,\phi Y)Z,\phi W)=g(\phi R(Z,W)X,\phi Y)+g(AZ,\phi
X)g(AW,\phi Y)\bigskip \\ 
-g(AW,\phi X)g(AZ,\phi Y)-g(AZ,X)g(\phi AW,\phi Y)+g(AW,X)g(\phi AZ,\phi
Y)\bigskip \\ 
-\eta (X)g(\phi R(Z,W)\xi ,\phi Y)-\eta (R(\phi X,\phi Y)Z)\eta (W).%
\end{array}
\label{3a15}
\end{equation}%
Putting $X=\phi X$ and $Y=\phi Y$ in (\ref{3a9}) we have,%
\begin{equation}
\begin{array}{l}
g(R(\phi X,\phi Y)\phi Z,\phi W)-g(\phi R(\phi X,\phi Y)Z,\phi W)\bigskip \\ 
=g(A\phi X,\phi Z)g(A\phi Y,\phi W)-g(A\phi Y,\phi Z)g(A\phi X,\phi
W)\bigskip \\ 
-g(A\phi X,Z)g(\phi A\phi Y,\phi W)+g(A\phi Y,Z)g(\phi A\phi X,\phi
W)\bigskip \\ 
-\eta (Z)g(\phi R(\phi X,\phi Y)\xi ,\phi W).%
\end{array}
\label{3a16}
\end{equation}%
Substitution of the (\ref{3a15}) into (\ref{3a16}) yields immediately%
\begin{eqnarray}
g(R(\phi X,\phi Y)\phi Z,\phi W) &=&g(\phi R(Z,W)X,\phi Y)+g(AZ,\phi
X)g(AW,\phi Y)\bigskip  \notag \\
&&-g(AW,\phi X)g(AZ,\phi Y)-g(AZ,X)g(\phi AW,\phi Y)\bigskip  \notag \\
&&+g(AW,X)g(\phi AZ,\phi Y)-\eta (X)g(\phi R(Z,W)\xi ,\phi Y)\bigskip
\label{3a17} \\
&&-\eta (R(\phi X,\phi Y)Z)\eta (W)+g(A\phi X,\phi Z)g(A\phi Y,\phi
W)\bigskip  \notag \\
&&-g(A\phi Y,\phi Z)g(A\phi X,\phi W)-g(A\phi X,Z)g(\phi A\phi Y,\phi
W)\bigskip  \notag \\
&&+g(A\phi Y,Z)g(\phi A\phi X,\phi W)-\eta (Z)g(\phi R(\phi X,\phi Y)\xi
,\phi W).  \notag
\end{eqnarray}%
By using (\ref{*}), the relation (\ref{3a17}) can be written as%
\begin{eqnarray}
g(R(\phi X,\phi Y)\phi Z,\phi W) &=&g(R(Z,W)X,Y)-\eta \left( R(Z,W)X\right)
\eta \left( Y\right)  \notag \\
&&-g(AZ,X)g(AW,Y)+g(AW,X)g(AZ,Y)\bigskip  \notag \\
&&-\eta (X)g(R(Z,W)\xi ,Y)-\eta (R(\phi X,\phi Y)Z)\eta (W)\bigskip  \notag
\\
&&+g(A\phi X,\phi Z)g(A\phi Y,\phi W)-g(A\phi Y,\phi Z)g(A\phi X,\phi
W)\bigskip  \label{3a19} \\
&&-\eta (Z)g(\phi R(\phi X,\phi Y)\xi ,\phi W).  \notag
\end{eqnarray}%
Substituting $Y=Z=e_{i}$ in (\ref{3a19}), summing over $i=1,2,...,2n+1$, and
using the relation $tr(A)=-2\alpha n$, it is not hard to prove 
\begin{eqnarray*}
-\phi Q\phi X-QX &=&-\phi l\phi X-lX+4\alpha (1-n)A \\
&&-4\alpha ^{2}(1-n)\phi ^{2}X-\eta (X)Q\xi +\dsum\limits_{i=1}^{2n+1}\eta
\left( R(\phi X,\phi e_{i})e_{i}\right) \xi
\end{eqnarray*}%
The rest of the proof follows acting $\phi $ on the last equation. Also, the
last equation reduces to the following formula 
\begin{equation}
Q\phi -\phi Q=l\phi -\phi l+4\alpha (1-n)\phi (\alpha \phi ^{2}+\phi
h)+4\alpha ^{2}(n-1)\phi X,  \label{7.91}
\end{equation}%
for all vector fields on $M^{2n+1}.$
\end{proof}

\begin{proposition}
Let $(M^{2n+1},\phi ,\xi ,\eta ,g)$ be an almost $\alpha $-cosymplectic $%
(\kappa ,\mu ,\nu )$-space. Then the following relation is true%
\begin{equation}
Q\phi -\phi Q=2\mu h\phi +2(2\alpha (n-1)+\nu )h,  \label{7.37}
\end{equation}%
for all vector fields on $M^{2n+1}.$
\end{proposition}

\begin{proof}
In view of Eq. (\ref{7.91}) and (\ref{7.29}), we find%
\begin{equation*}
Q\phi -\phi Q=2\mu h\phi +2\nu h-4\alpha (1-n)h,
\end{equation*}%
which proves the required result.
\end{proof}

\begin{theorem}
Let $(M^{2n+1},\phi ,\xi ,\eta ,g)$ be an almost $\alpha $-cosymplectic
manifold, $\nabla _{\xi }h=0$, the sectional curvatures of all plane
sections $\xi $ are $\neq \alpha ^{2}$ at some point and the characteristic
vector space of $h$ is completely formed by the direct sum $D(\lambda
)\oplus D(-\lambda )$ on $\mathcal{D}$. If the tensor field $h$ is $\eta $%
-parallel, then $(M^{2n+1},\phi ,\xi ,\eta ,g)$ is a $(\kappa ,0,2\alpha )$%
-space with $\kappa =-\left( \alpha ^{2}+\lambda ^{2}\right) .$
\end{theorem}

\begin{proof}
Using (\ref{3.1}) we obtain%
\begin{eqnarray}
R(X,Y)\xi &=&\alpha ^{2}\left[ \eta (X)Y-\eta (Y)X\right] -\alpha \left[
\eta (X)\phi hY-\eta (Y)\phi hX\right]  \label{7.43} \\
&&+(\nabla _{X}h)\phi Y-(\nabla _{Y}h)\phi X+h((\nabla _{X}\phi )Y-(\nabla
_{Y}\phi )X),  \notag
\end{eqnarray}%
by acting the formula $(\nabla _{X}\phi h)Y=-(\nabla _{X}h)\phi Y-h(\nabla
_{X}\phi )Y$ for any vector fields $X,Y.$ Since $h$ is $\eta $-parallel,
substituting $Y=\phi Y$ in Eq. (\ref{4.1}) we find%
\begin{equation}
(\nabla _{X}h)\phi Y=g(\phi Y,\alpha hX+\phi h^{2}X)\xi .  \label{7.44}
\end{equation}%
Then replacing $X$ by $Y$ in Eq. (\ref{7.44})and by virtue of Eqs. (\ref%
{7.43}) and (\ref{7.44}), we get%
\begin{equation}
(\nabla _{X}h)\phi Y-(\nabla _{Y}h)\phi X=0.  \label{7.45}
\end{equation}%
Using (\ref{4.16}) and combining (\ref{7.45}) in Eq. (\ref{7.43}), we also
get%
\begin{eqnarray}
R(X,Y)\xi &=&-\alpha ^{2}\left[ \eta (Y)X-\eta (X)Y\right] +2\alpha \left[
\eta (Y)\phi hX-\eta (X)\phi hY\right]  \label{7.46} \\
&&-[\eta (Y)h^{2}X-\eta (X)h^{2}Y].  \notag
\end{eqnarray}%
Furthermore, Eq. (\ref{7.47}) is valid because of the our assumption. At
that rate, Eq. (\ref{7.46}) reduces to%
\begin{equation*}
R(X,Y)\xi =-(\alpha ^{2}+\lambda ^{2})\left[ \eta (Y)X-\eta (X)Y\right]
+2\alpha \left[ \eta (Y)\phi hX-\eta (X)\phi hY\right] ,
\end{equation*}%
by using Eq. (\ref{7.47}). The proof is completed.
\end{proof}

\begin{theorem}
The following differential equation is valid on every $(2n+1)$-dimensional
almost $\alpha $-cosymplectic $(\kappa ,\mu ,\nu )$-space $(M^{2n+1},\phi
,\xi ,\eta ,g):$%
\begin{eqnarray}
0 &=&\xi (\kappa )(\eta (Y)X-\eta (X)Y)+\xi (\mu )(\eta (Y)hX-\eta
(X)hY)+\xi (\nu )(\eta (Y)\phi hX  \notag \\
&&-\eta (X)\phi hY)-X(\kappa )\phi ^{2}Y+X(\mu )hY+X(\nu )\phi hY-Y(\mu
)hX-Y(\nu )\phi hX  \notag \\
&&+Y(\kappa )\phi ^{2}X+2(\kappa +\alpha ^{2})\mu g(\phi X,Y)\xi +2\mu
g(hX,\phi hY)\xi .  \label{7.27}
\end{eqnarray}
\end{theorem}

\begin{proof}
Differentiating the formula (\ref{501}) along an vector field $Z$ we have%
\begin{eqnarray*}
(\nabla _{Z}R)(X,Y)\xi &=&Z(\kappa )\left[ \eta (Y)X-\eta (X)Y\right] +Z(\mu
)\left[ \eta (Y)hX-\eta (X)hY\right] \\
&&+Z(\nu )\left[ \eta (Y)\phi hX-\eta (X)\phi hY\right] +\kappa \left[ \eta
(\nabla _{Z}Y)X+g(Y,\nabla _{Z}\xi )X\right. \\
&&\left. +\eta (Y)\nabla _{Z}X\right] +\kappa \left[ -\eta (\nabla
_{Z}X)Y-g(X,\nabla _{Z}\xi )Y-\eta (X)\nabla _{Z}Y\right] \\
&&+\mu \left[ \eta (\nabla _{Z}Y)hX+g(Y,\nabla _{Z}\xi )hX\right] \\
&&+\mu \left[ \eta (Y)\nabla _{Z}hX-\eta (\nabla _{Z}X)hY-g(X,\nabla _{Z}\xi
)hY-\eta (X)\nabla _{Z}hY\right] \\
&&+\nu \left[ \eta (\nabla _{Z}Y)\phi hX+g(Y,\nabla _{Z}\xi )\phi hX+\eta
(Y)\nabla _{Z}\phi hX\right] \\
&&+\nu \left[ -\eta (\nabla _{Z}X)\phi hY-g(X,\nabla _{Z}\xi )\phi hY-\eta
(X)\nabla _{Z}\phi hY\right] ,
\end{eqnarray*}%
by considering the equation%
\begin{equation*}
(\nabla _{Z}R)(X,Y)\xi =\nabla _{Z}R(X,Y)\xi -R(\nabla _{Z}X,Y)\xi
-R(X,\nabla _{Z}Y)\xi -R(X,Y)\nabla _{Z}\xi .
\end{equation*}%
Then by using (\ref{2.3}) we also have%
\begin{eqnarray*}
(\nabla _{Z}R)(X,Y)\xi &=&Z(\kappa )\left[ \eta (Y)X-\eta (X)Y\right] +Z(\mu
)\left[ \eta (Y)hX-\eta (X)hY\right] \\
&&+Z(\nu )\left[ \eta (Y)\phi hX-\eta (X)\phi hY\right] +\kappa \left[ \eta
(\nabla _{Z}Y)X\right] \\
&&+\kappa \left[ \alpha g(\phi Y,\phi Z)X-g(Y,\phi hZ)X+g(X,\phi hZ)Y\right]
\end{eqnarray*}%
\begin{eqnarray*}
&&+\kappa \left[ \eta (Y)\nabla _{Z}X-\eta (\nabla _{Z}X)Y-\alpha g(\phi
X,\phi Z)Y-\eta (X)\nabla _{Z}Y\right] \\
&&+\mu \left[ \eta (\nabla _{Z}Y)hX+\alpha g(\phi Y,\phi Z)hX-g(Y,\phi hZ)hX%
\right] \\
&&+\mu \left[ \eta (Y)(\nabla _{Z}h)X-\eta (\nabla _{Z}X)hY-\alpha g(\phi
X,\phi Z)hY\right] \\
&&+\mu \left[ g(X,\phi hZ)hY-\eta (X)(\nabla _{Z}Y)Y\right] +\nu \left[
\alpha g(\phi Y,\phi Z)\phi hX\right] \\
&&+\nu \left[ \eta (\nabla _{Z}Y)\phi hX-g(Y,\phi hZ)\phi hX+\eta (Y)(\nabla
_{Z}h)\phi X\right] \\
&&+\nu \left[ -\eta (\nabla _{Z}X)\phi hY-\alpha g(\phi X,\phi Z)\phi
hY-\eta (X)(\nabla _{Z}Y)\phi Y\right] \\
&&+\nu \left[ g(X,\phi hZ)\phi hY\right] -\kappa \left[ \eta (Y)\nabla
_{Z}X-\eta (\nabla _{Z}X)Y\right] \\
&&-\mu \left[ \eta (Y)h\nabla _{Z}X-\eta (\nabla _{Z}X)hY\right] -\kappa %
\left[ -\eta (X)\nabla _{Z}Y+\eta (\nabla _{Z}Y)X\right] \\
&&-\mu \left[ -\eta (X)h\nabla _{Z}Y+\eta (\nabla _{Z}Y)hX\right] -\nu \left[
-\eta (X)\phi h\nabla _{Z}Y+\eta (\nabla _{Z}Y)\phi hX\right] \\
&&-\nu \left[ \eta (Y)\phi h\nabla _{Z}X-\eta (\nabla _{Z}X)\phi hY\right]
+\alpha \kappa \eta (Z)\left[ \eta (Y)X-\eta (X)Y\right] \\
&&+\alpha \mu \eta (Z)\left[ \eta (Y)hX-\eta (X)hY\right] +\alpha \nu \eta
(Z)\left[ \eta (Y)\phi hX-\eta (X)\phi hY\right] \\
&&-\alpha R(X,Y)Z+R(X,Y)\phi hZ.
\end{eqnarray*}%
Now, using again Eqs. (\ref{2.3}) and (\ref{501}) the last equation reduces
to%
\begin{eqnarray*}
(\nabla _{Z}R)(X,Y)\xi &=&Z(\kappa )\left[ \eta (Y)X-\eta (X)Y\right] +Z(\mu
)\left[ \eta (Y)hX-\eta (X)hY\right] \\
&&+Z(\nu )\left[ \eta (Y)\phi hX-\eta (X)\phi hY\right] +\kappa \left[
\alpha g(Z,X)Y\right] \\
&&+\kappa \left[ -\alpha g(X,Z)Y+g(X,\phi hZ)Y-g(Y,\phi hZ)X\right] \\
&&+\mu \left[ -g(Y,\phi hZ)hX+\eta (Y)(\nabla _{Z}h)X+\alpha g(Y,Z)hX\right]
\\
&&+\mu \left[ -\alpha g(X,Z)hY+g(X,\phi hZ)hY-\eta (X)(\nabla _{Z}h)Y\right]
\\
&&+\nu \left[ \alpha g(Y,Z)\phi hX-g(Y,\phi hZ)\phi hX+\eta (Y)(\nabla
_{Z}\phi h)X\right] \\
&&+\nu \left[ -\alpha g(X,Z)\phi hY+g(X,\phi hZ)\phi hY-\eta (X)(\nabla
_{Z}\phi h)Y\right] \\
&&-\alpha R(X,Y)Z+R(X,Y)\phi hZ.
\end{eqnarray*}%
Next using the last equation and the second Bianchi identity%
\begin{equation*}
(\nabla _{Z}R)(X,Y)\xi +(\nabla _{X}R)(Y,Z)\xi +(\nabla _{Y}R)(Z,X)\xi =0,
\end{equation*}%
we obtain%
\begin{eqnarray*}
0 &=&Z(\kappa )\left[ \eta (Y)X-\eta (X)Y\right] +Z(\mu )\left[ \eta
(Y)hX-\eta (X)hY\right] \\
&&+Z(\nu )\left[ \eta (Y)\phi hX-\eta (X)\phi hY\right] +X(\kappa )\left[
\eta (Z)Y-\eta (Y)Z\right] \\
&&+X(\mu )\left[ \eta (Z)hY-\eta (Y)hZ\right] +X(\nu )\left[ \eta (Z)\phi
hY-\eta (Y)\phi hZ\right] \\
&&+Y(\kappa )\left[ \eta (X)Z-\eta (Z)X\right] +Y(\mu )\left[ \eta
(X)hZ-\eta (Z)hX\right] \\
&&+Y(\nu )\left[ \eta (X)\phi hZ-\eta (Z)\phi hX\right] +\mu \left[ \eta
(Y)\left( (\nabla _{Z}h)X-(\nabla _{X}h)Z\right) \right] \\
&&+\mu \left[ \eta (Z)\left( (\nabla _{X}h)Y-(\nabla _{Y}h)X\right) +\eta
(X)\left( (\nabla _{Y}h)Z-(\nabla _{Z}h)Y\right) \right] \\
&&+\nu \left[ \eta (Y)\left( (\nabla _{Z}\phi h)X-(\nabla _{X}\phi
h)Z\right) +\eta (Z)\left( (\nabla _{X}\phi h)Y-(\nabla _{Y}\phi h)X\right) %
\right] \\
&&+\nu \left[ \eta (X)\left( (\nabla _{Y}\phi h)Z-(\nabla _{Z}\phi
h)Y\right) \right] +R(X,Y)\phi hZ+R(Y,Z)\phi hX \\
&&-\alpha \left[ R(X,Y)Z+R(Y,Z)X+R(Z,X)Y\right] +R(Z,X)\phi hY,
\end{eqnarray*}%
for all vector fields $X,Y,Z.$ Putting $\xi $ instead of $Z$ in the above
equation, we obtain%
\begin{eqnarray*}
0 &=&\xi (\kappa )\left[ \eta (Y)X-\eta (X)Y\right] +\xi (\mu )\left[ \eta
(Y)hX-\eta (X)hY\right] \\
&&+\xi (\nu )\left[ \eta (Y)\phi hX-\eta (X)\phi hY\right] -X(\kappa )\phi
^{2}Y+X(\mu )hY \\
&&+X(\nu )\phi hY+Y(\kappa )\phi ^{2}X-Y(\mu )hX-Y(\nu )\phi hX \\
&&+\mu \eta (Y)\left[ -(\kappa +\alpha ^{2})\phi X-\mu \phi hX-(\alpha -\nu
)hX\right] \\
&&+\mu (\kappa +\alpha ^{2})\left[ \eta (Y)\phi X-\eta (X)\phi Y+2g(\phi
X,Y)\xi \right] \\
&&+\mu ^{2}\left[ \eta (Y)\phi hX-\eta (X)\phi hY\right] +\mu (\alpha -\nu )%
\left[ \eta (Y)hX-\eta (X)hY\right] \\
&&+\mu \eta (X)\left[ (\kappa +\alpha ^{2})\phi Y+\mu h\phi Y+(\alpha -\nu
)hY\right] \\
&&+\nu \eta (Y)\left[ -(\kappa +\alpha ^{2})\phi ^{2}X+\mu hX-(\alpha -\nu
)\phi hX\right] \\
&&-\nu (\kappa +\alpha ^{2})\left[ \eta (Y)X-\eta (X)Y\right] -\nu \mu \left[
\eta (Y)hX-\eta (X)hY\right] \\
&&+\nu (\alpha -\nu )\left[ \eta (Y)\phi hX-\eta (X)\phi hY\right] +\nu \eta
(X)(\kappa +\alpha ^{2})\phi ^{2}Y \\
&&+\nu \eta (X)\left[ -\mu hY+(\alpha -\nu )\phi hY\right] -R(\xi ,Y)\phi
hX+R(\xi ,X)\phi hY.
\end{eqnarray*}%
Finally, if we substitute (\ref{7.34}) (\ref{7.25}) and (\ref{7.26}) in the
last equation, then we deduce Eq. (\ref{7.27}).
\end{proof}

By the means of the proof of $[$Koufogiorgos et al., Lemma $4.4],$ we have

\begin{lemma}
Let $(M^{2n+1},\phi ,\xi ,\eta ,g)$ be an almost $\alpha $-cosymplectic $%
(\kappa ,\mu ,\nu )$-space. For every $p\in N,$ there exists neighborhood $W$
of $p$ and orthonormal local vector fields $X_{i},\phi X_{i}$ and $\xi $ for 
$i=1,\ldots ,n,$ defined on $W$, such that%
\begin{equation}
hX_{i}=\lambda X_{i},\text{ \ }h\phi X_{i}=-\lambda X_{i},\text{ \ }h\xi =0,
\label{7.38}
\end{equation}%
for $i=1,\ldots ,n,$ where $\lambda =\sqrt{-\left( \kappa +\alpha
^{2}\right) }.$
\end{lemma}

\begin{proof}
The proof of that lemma is similiar to that of [Koufogiorgos et al. 2008,
Lemma $4.2$].
\end{proof}

Now, we will explain why the functions $\kappa ,\mu $ and $\nu $ are element
of $\mathcal{R}_{\eta }(M^{2n+1}).$ Using Lemma $1$ we will prove the
following theorem.

\begin{theorem}
Let $(M^{2n+1},\phi ,\xi ,\eta ,g)$ be an almost $\alpha $-cosymplectic $%
(\kappa ,\mu ,\nu )$-space with the dimension greater than $3.$ Then the
functions $\kappa ,\mu $ and $\nu $ are element of $\mathcal{R}_{\eta
}(M^{2n+1}).$
\end{theorem}

\begin{proof}
Lemma $1$ implies that the existence of a local orthonormal basis $\left\{
X_{i},\phi X_{i},\xi \right\} $ such that%
\begin{equation*}
he_{i}=\lambda e_{i},\text{ }h\phi e_{i}=-\lambda \phi e_{i},\text{ }h\xi =0,%
\text{ }\lambda =\sqrt{-\left( \kappa +\alpha ^{2}\right) },
\end{equation*}%
on $W.$ Substituting $X=e_{i}$ and $Y=\phi e_{i}$ in Eq. (\ref{7.27}), we
obtain that%
\begin{equation*}
e_{i}(\kappa )\phi e_{i}-\lambda e_{i}(\mu )\phi e_{i}+\lambda e_{i}(\nu
)e_{i}-\lambda \phi e_{i}(\mu )e_{i}-\lambda \phi e_{i}(\nu )\phi e_{i}-\phi
e_{i}(\kappa )e_{i}=0.
\end{equation*}%
This means that%
\begin{equation*}
\left[ e_{i}(\kappa )-\lambda e_{i}(\mu )-\lambda \phi e_{i}(\nu )\right]
\phi e_{i}+\left[ \lambda e_{i}(\nu )-\lambda \phi e_{i}(\mu )-\phi
e_{i}(\kappa )\right] =0.
\end{equation*}%
Since $\left\{ e_{i},eX_{i}\right\} $ is linearly independent, we have%
\begin{equation}
\begin{array}{c}
e_{i}(\kappa )-\lambda e_{i}(\mu )-\lambda \phi e_{i}(\nu )=0, \\ 
\lambda e_{i}(\nu )-\lambda \phi e_{i}(\mu )-\phi e_{i}(\kappa )=0.%
\end{array}
\label{7.48}
\end{equation}%
In addition, replacing $X$ and $Y\ $by $e_{i}$ and $e_{j}$, respectively,
for $i\neq j$, Eq. (\ref{7.27}) provides that%
\begin{equation}
\begin{array}{c}
e_{i}(\kappa )+\lambda e_{i}(\mu )=0, \\ 
e_{i}(\nu )=0.%
\end{array}
\label{7.50}
\end{equation}%
Besides, substituting $X=\phi e_{i}$ and $Y=\phi e_{j}$, for $i\neq j$, in
Eq. (\ref{7.27}), we get%
\begin{equation}
\begin{array}{c}
\phi e_{i}(\kappa )-\lambda \phi e_{i}(\mu )=0, \\ 
\phi e_{i}(\nu )=0.%
\end{array}
\label{7.49}
\end{equation}%
In view of Eqs. (\ref{7.48}), (\ref{7.50}) and (\ref{7.49}), we deduce%
\begin{equation*}
e_{i}(\kappa )=e_{i}(\mu )=e_{i}(\nu )=\phi e_{i}(\kappa )=\phi e_{i}(\mu
)=\phi e_{i}(\nu )=0.
\end{equation*}%
For an arbitrary function $\kappa ,$ we obtain that $d\kappa =\xi (\kappa
)\eta $ in the last equation system. In this way, \ we can write%
\begin{equation}
0=d^{2}\kappa =d(d\kappa )=d\xi (\kappa )\wedge \eta +\xi (\kappa )d\eta .
\label{7.51}
\end{equation}%
Since $d\eta =0,$ acting the last equation, it gives that $d\xi (\kappa
)\wedge \eta =0.$ Namely, the function $\kappa $ is an element of $\mathcal{R%
}_{\eta }(M^{2n+1})$ on every connected component of $W.$ Analogously, it
can be shown that the functions $\mu $ and $\nu $ are also non-constant on
every connected component of $W.$
\end{proof}

\begin{corollary}
The functions $\kappa ,\mu $ and $\nu $ are constants if and only if these
functions are constants along the characteristic vector field $\xi $ for
almost $\alpha $-cosymplectic $(\kappa ,\mu ,\nu )$-space $(M^{2n+1},\phi
,\xi ,\eta ,g)$ with $n>1.$
\end{corollary}

\section{The existence of non-constant $(\protect\kappa ,\protect\mu ,%
\protect\nu )$-spaces in dimension $3$}

In this part, we will show that the existence of almost $\alpha $%
-cosymplectic $(\kappa ,\mu ,\nu )$-space for the non-constant functions $%
\kappa ,\mu $ and $\nu $ in dimension $3.$

Let $U$ be the open subset of $M^{3}$ where the tensor field $h\neq 0$ and
let $U^{\prime }$ be the open subset of points $p\in M^{3}$ such that $h=0$
in a neighborhood of $p.$ Thus the association set of $U\cup U^{\prime }$ is
an open and dense subset of $M^{3}.$ For every $p\in U$ there exists an open
neighborhood of $p$ such that $he=\lambda e$ and $h\phi e=-\lambda \phi e,$
where $\lambda $ is a positive non-vanishing smooth function. So every
properties satisfying on $U\cup U^{\prime }$ is valid on $M^{3}.$ Therefore,
there exists a local orthonormal basis $\left\{ e,\phi e,\xi \right\} $ of
smooth eigen functions of $h$ in a neighborhood of $p$ for every point $p\in
U\cup U^{\prime }.$ This basis is called $\phi $-basis. The following lemma
will be useful for the latter case.

\begin{lemma}
Let $(M^{3},\phi ,\xi ,\eta ,g)$ be an almost $\alpha $-cosymplectic
manifold. Then for the covariant derivative on $U$ the following equations
are valid%
\begin{equation*}
\begin{array}{llllll}
\nabla _{\xi }e & = & -a\phi e, & \nabla _{\xi }\phi e & = & ae, \\ 
\nabla _{e}\xi & = & \alpha e-\lambda \phi e, & \nabla _{\phi e}\xi & = & 
-\lambda e+\alpha \phi e, \\ 
\nabla _{e}e & = & b\phi e-\alpha \xi , & \nabla _{\phi e}\phi e & = & 
ce-\alpha \xi , \\ 
\nabla _{e}\phi e & = & -be+\lambda \xi , & \nabla _{\phi e}e & = & -c\phi
e+\lambda \xi ,%
\end{array}%
\end{equation*}%
where $a$ is a smooth function, $b=g(\nabla _{e}e,\phi e)$ and $c=g(\nabla
_{\phi e}\phi e,e)$ defined by%
\begin{equation*}
b=\frac{1}{2\lambda }\left[ (\phi e)(\lambda )+\sigma (e)\right] ,~\ \sigma
(e)=S(\xi ,e)=g(Q\xi ,e),
\end{equation*}%
and%
\begin{equation*}
c=\frac{1}{2\lambda }\left[ e(\lambda )+\sigma (\phi e)\right] ,~\ \sigma
(\phi e)=S(\xi ,\phi e)=g(Q\xi ,\phi e),
\end{equation*}%
respectively.
\end{lemma}

\begin{proof}
Replacing $X$ by $e$ and $\phi e$ in Eq. (\ref{2.3})$,$ respectively, we find%
\begin{equation*}
\begin{array}{llllll}
\nabla _{e}\xi & = & -\alpha \phi ^{2}e-\phi he, & \nabla _{\phi e}\xi & = & 
-\alpha \phi ^{3}e-\phi h\phi e, \\ 
& = & \alpha e-\alpha \eta (e)\xi +h\phi e, &  & = & \alpha \phi e-he, \\ 
& = & \alpha e-\lambda \phi e, &  & = & \alpha \phi e-\lambda e,%
\end{array}%
\end{equation*}%
for any vector field $X.$ Also, we have%
\begin{eqnarray*}
\nabla _{\xi }e &=&g(\nabla _{\xi }e,e)e+g(\nabla _{\xi }e,\phi e)\phi
e+g(\nabla _{\xi }e,\xi )\xi \\
&=&-g(e,\nabla _{\xi }\phi e)\phi e,
\end{eqnarray*}%
where $a$ is defined by $a=g(e,\nabla _{\xi }\phi e).$ So $\nabla _{\xi }e$
is obtained by the formula $\nabla _{\xi }e=-a\phi e.$ Following this
procedure, the other covariant derivative equalities can easily find. We
recall that the curvature tensor $3-$dimensional Riemannian manifold is
given by 
\begin{equation}
\begin{array}{l}
R(X,Y)Z=-S(X,Z)Y+S(Y,Z)X-g(X,Z)QY \\ 
+g(Y,Z)QX+\frac{r}{2}[g(X,Z)Y-g(Y,Z)X],%
\end{array}
\label{7.52b}
\end{equation}%
with the dimensionalthree case, for any vector fields $X,Y,Z.$ Putting $X=e,$
$Y=\phi e$ and $Z=\xi $ in the last equation, we obtain%
\begin{equation*}
R(e,\phi e)\xi =-g(Qe,\xi )\phi e+g(Q\phi e,\xi )e.
\end{equation*}%
Since $\sigma (X)=g(Q\xi ,X),$ we have%
\begin{equation}
R(e,\phi e)\xi =-\sigma (e)\phi e+\sigma (\phi e)e,  \label{7.53}
\end{equation}%
for any vector field $X$. By using the curvature properties of the
Riemannian tensor, we also have%
\begin{eqnarray}
R(e,\phi e)\xi &=&(\nabla _{\phi e}\phi h)e-(\nabla _{e}\phi h)\phi e, 
\notag \\
&=&(2\lambda c-e(\lambda ))e+(-2\lambda b+(\phi e)(\lambda ))\phi e.
\label{7.52}
\end{eqnarray}%
In this case, combining (\ref{7.53}) and (\ref{7.52}), we deduce that%
\begin{equation*}
\sigma (e)=2\lambda b-(\phi e)(\lambda ),\text{ }\sigma (\phi e)=2\lambda
c-e(\lambda ).
\end{equation*}%
Hence, the functions $b$ and $c$ are obtained by the above relations.
\end{proof}

\begin{proposition}
Let $(M^{3},\phi ,\xi ,\eta ,g)$ be an almost $\alpha $-cosymplectic
manifold. On $U$ the following relation is true:%
\begin{equation}
\nabla _{\xi }h=2ah\phi +\xi (\lambda )s,  \label{7.54}
\end{equation}%
where $s$ is the tensor field of type $(1,1)$ defined by $s\xi =0,$ $se=e$
and $s\phi e=-\phi e.$
\end{proposition}

\begin{proof}
First, we check the tensor field $h$ which differentiating along $\xi .$ In
that case, we have%
\begin{equation*}
(\nabla _{\xi }h)e=-2\lambda a\phi e+\xi (\lambda )e,\text{ }(\nabla _{\xi
}h)\phi e=-2\lambda ae-\xi (\lambda )\phi e.
\end{equation*}%
In addition, we also have $(\nabla _{\xi }h)\xi =0.$ Then we obtain (\ref%
{7.54}) with the help of the last equations. It is clear that $tr(s)=0.$
\end{proof}

\begin{remark}
Since $h=0$ on the open subset $U^{\prime },$ we have $\xi (\lambda
)s=\nabla _{\xi }h=0.$
\end{remark}

\begin{proposition}
Let $(M^{3},\phi ,\xi ,\eta ,g)$ be an almost $\alpha $-cosymplectic
manifold. Then the integral submanifold of the distribution $\mathcal{D}$ on 
$M^{3}$ has Kaehlerian structures if and only if the following relation is
true%
\begin{equation*}
(\nabla _{X}\phi )Y=g(\alpha \phi X+hX,Y)\xi -\eta (Y)(\alpha \phi X+hX),
\end{equation*}%
for arbitrary vector fields $X,Y$ on $M^{3}.$
\end{proposition}

\begin{proof}
The $3$-form $\eta \wedge \Phi $ is equal to the volume element of $M^{3}.$
Since every volume element is constant, it gives $\nabla _{X}(\eta \wedge
\Phi )=0,$ for any vector field $X.$ By the means of this equation, we obtain%
\begin{eqnarray*}
0 &=&(\nabla _{X}\eta )(Y)\Phi (Z,W)+\eta (Y)(\nabla _{X}\Phi )(Z,W) \\
&&+(\nabla _{X}\eta )(Z)\Phi (W,Y)+\eta (Z)(\nabla _{X}\Phi )(W,Y) \\
&&+(\nabla _{X}\eta )(W)\Phi (Y,Z)+\eta (W)(\nabla _{X}\Phi )(Y,Z).
\end{eqnarray*}%
Setting $\xi $ instead of $W$ in the last equation, we have%
\begin{equation}
(\nabla _{X}\Phi )(Z,Y)=-\eta (Z)(\nabla _{X}\Phi )(Y,\xi )+\eta (Y)(\nabla
_{X}\Phi )(Z,\xi ).  \label{7.55}
\end{equation}%
In view of Eq. (\ref{7.55}), we deduce that%
\begin{equation*}
g(Z,(\nabla _{X}\phi )Y)=g(Z,g(\phi \nabla _{X}\xi ,Y)\xi -\eta (Y)\phi
\nabla _{X}\xi ),
\end{equation*}%
which yields%
\begin{equation*}
(\nabla _{X}\phi )Y=g(\phi \nabla _{X}\xi ,Y)\xi -\eta (Y)\phi \nabla
_{X}\xi ,
\end{equation*}%
for arbitrary vector fields $X$ and $Y.$
\end{proof}

\begin{proposition}
Let $(M^{3},\phi ,\xi ,\eta ,g)$ be an almost $\alpha $-cosymplectic
manifold. Then the following relation is satisfied on $M^{3}.$%
\begin{equation}
h^{2}-\alpha ^{2}\phi ^{2}=\frac{tr(l)}{2}\phi ^{2}.  \label{7.56}
\end{equation}
\end{proposition}

\begin{proof}
By using (\ref{3.6}), we get $tr(l)=-2\left[ \alpha ^{2}+\lambda ^{2}\right]
,$ for all vector fields on $M^{3}.$ Besides, we calculate the statement of $%
h^{2}-\alpha ^{2}\phi ^{2}$ with respect to the basis components$,$ then we
obtain that%
\begin{equation*}
h^{2}e-\alpha ^{2}\phi ^{2}e=\frac{tr(l)}{2}\phi ^{2}e,\text{ }h^{2}\phi
e-\alpha ^{2}\phi ^{3}e=\frac{tr(l)}{2}\phi ^{2}\phi e.
\end{equation*}%
Also, it has been obviously seen that $h^{2}\xi -\alpha ^{2}\phi ^{2}\xi =%
\dfrac{tr(l)}{2}\phi ^{2}\xi =0.$ These equations completes the proof of Eq.
(\ref{7.56}).
\end{proof}

\begin{lemma}
Let $(M^{3},\phi ,\xi ,\eta ,g)$ be an almost $\alpha $-cosymplectic
manifold. Then the Ricci operator $Q$ satisfies the relation%
\begin{eqnarray}
Q &=&\tilde{a}I+\tilde{b}\eta \otimes \xi +2\alpha \phi h+\phi (\nabla _{\xi
}h)-\sigma (\phi ^{2})\otimes \xi  \label{7.58} \\
&&+\sigma (e)\eta \otimes e+\sigma (\phi e)\eta \otimes \phi e,  \notag
\end{eqnarray}%
where $\tilde{a}$ and $\tilde{b}$ are smooth functions defined by $\tilde{a}=%
\frac{1}{2}r+\alpha ^{2}+\lambda ^{2}$ and $\tilde{b}=-\frac{1}{2}r-3\alpha
^{2}-3\lambda ^{2},$ respectively.
\end{lemma}

\begin{proof}
For $3$-dimensional case, we deduce that%
\begin{equation*}
lX=tr(l)X-S(X,\xi )\xi +QX-\eta (X)Q\xi -\frac{r}{2}\left( X-\eta (X)\xi
\right) ,
\end{equation*}%
for any vector field $X.$ The above equation implies%
\begin{eqnarray*}
QX &=&\alpha ^{2}\phi ^{2}X+2\alpha \phi hX-h^{2}X+\phi (\nabla _{\xi
}h)X-tr(l)X \\
&&-S(X,\xi )\xi +\eta (X)Q\xi +\frac{r}{2}\left( X-\eta (X)\xi \right) .
\end{eqnarray*}%
Otherwise, since $S(X,\xi )=-S(\phi ^{2}X,\xi )+\eta (X)tr(l),$ we have%
\begin{eqnarray}
QX &=&-\frac{tr(l)}{2}\phi ^{2}X+2\alpha \phi hX+\phi (\nabla _{\xi
}h)X-tr(l)X  \label{7.57} \\
&&-S(\phi ^{2}X,\xi )\xi +\eta (X)tr(l)\xi +\eta (X)Q\xi -\frac{r}{2}\phi
^{2}X.  \notag
\end{eqnarray}%
Thus it is clear that $Q\xi =\sigma (e)e+\sigma (\phi e)\phi e+tr(l)\xi .$
Acting the last equation in Eq. (\ref{7.57}), we obtain%
\begin{eqnarray*}
QX &=&\left[ \frac{1}{2}r+\alpha ^{2}+\lambda ^{2}\right] X+\left[ -\frac{1}{%
2}r-3\alpha ^{2}-3\lambda ^{2}\right] \eta (X)\xi \\
&&+2\alpha \phi hX+\phi (\nabla _{\xi }h)X-S(\phi ^{2}X,\xi )\xi \\
&&+\eta (X)\sigma (e)e+\eta (X)\sigma (\phi e)\phi e,
\end{eqnarray*}%
for arbitrary vector field $X.$ Therefore, we find the functions $\tilde{a}$
and $\tilde{b}$ mentioned in Eq. (\ref{7.58}).
\end{proof}

\begin{theorem}
Let $(M^{3},\phi ,\xi ,\eta ,g)$ be an almost $\alpha $-cosymplectic
manifold. If $\sigma \equiv 0,$ then the $(\kappa ,\mu ,\nu )$-structure
always exists on every open and dense subset of $M^{3}.$
\end{theorem}

\begin{proof}
Substituting $\sigma \equiv 0$ and $s=\frac{1}{\lambda }h$ in Eq. (\ref{7.58}%
) we deduce%
\begin{equation}
Q=\tilde{a}I+\tilde{b}\eta \otimes \xi +2ah+(2\alpha +\frac{\xi (\lambda )}{%
\lambda })\phi h,  \label{7.59}
\end{equation}%
which yields%
\begin{equation}
Q\xi =tr(l)\xi ,  \label{7.60}
\end{equation}%
for any vector fields on $M^{3}.$ Setting $\xi $ instead of $Z$ in Eq. (\ref%
{7.52b}) we obtain%
\begin{equation}
\begin{array}{l}
R(X,Y)\xi =-S(X,\xi )Y+S(Y,\xi )X+\eta (Y)QX \\ 
-\eta (X)QY-\frac{r}{2}[\eta (Y)X-\eta (X)Y],%
\end{array}
\label{7.70}
\end{equation}%
and replacing $X$ by $\xi $ , then we find $Q\xi =tr(l).$ So the last
equation shows that%
\begin{equation}
S(Y,\xi )=tr(l)\eta (Y),  \label{7.62}
\end{equation}%
for any vector field $Y.$ Thus by virtue of Eqs. (\ref{7.59}) (\ref{7.60})
and (\ref{7.62}), we get%
\begin{eqnarray*}
R(X,Y)\xi &=&-\left( \alpha ^{2}+\lambda ^{2}\right) (\eta (Y)X-\eta (X)Y) \\
&&+2a(\eta (Y)hX-\eta (X)hY) \\
&&+(2\alpha +\frac{\xi (\lambda )}{\lambda })(\eta (Y)\phi hX-\eta (X)\phi
hY),
\end{eqnarray*}%
where the functions $\kappa ,\mu $ and $\nu $ defined by $\kappa =\dfrac{%
tr(l)}{2},$ $\mu =2a$ and $\nu =2\alpha +\dfrac{\xi (\lambda )}{\lambda },$
respectively and it completes the proof of the Theorem.
\end{proof}

\begin{corollary}
By using Eq. (\ref{7.37}) for $3$-dimensional $(n=1),$ we have%
\begin{equation*}
Q\phi -\phi Q=2\mu h\phi +2\nu h.
\end{equation*}
\end{corollary}

Now, we investigate the inversion of the above corollary on the three
dimensional almost $\alpha $-cosymplectic manifolds.

\begin{theorem}
Let $(M^{3},\phi ,\xi ,\eta ,g)$ be an almost $\alpha $-cosymplectic
manifold. If the following relation is satisfied%
\begin{equation}
Q\phi -\phi Q=f_{1}h\phi +f_{2}h,  \label{7.63}
\end{equation}%
then the manifold $(M^{3},\phi ,\xi ,\eta ,g)$ is an almost $\alpha $%
-cosymplectic $(\kappa ,\mu ,\nu )$-space, where the functions $f_{1},$ $%
f_{2}\in C^{\infty }.$
\end{theorem}

\begin{proof}
Considering the equations (\ref{3.2}) and (\ref{7.52b} )we have%
\begin{equation}
\begin{array}{l}
\alpha ^{2}\phi ^{2}X+2\alpha \phi hX-h^{2}X+\phi (\nabla _{\xi }h)X \\ 
=QX-2tr(l)\eta (X)\xi +tr(l)X-\frac{r}{2}\left( X-\eta (X)\xi \right) .%
\end{array}
\label{7.64}
\end{equation}%
Applying $\phi $ both two sides of Eq. (\ref{7.64}), we get that%
\begin{equation}
-\alpha ^{2}\phi X-\phi h^{2}X-2\alpha hX-(\nabla _{\xi }h)X=\phi
QX+tr(l)\phi X-\frac{r}{2}\phi X.  \label{7.65}
\end{equation}%
On the other hand, replacing $X$ by $\phi X$ in Eq. (\ref{7.65}), we find%
\begin{equation}
\begin{array}{l}
-\alpha ^{2}\phi X+2\alpha hX-h^{2}\phi X+(\nabla _{\xi }h)X \\ 
=Q\phi X+tr(l)\phi X-\dfrac{r}{2}\phi X,%
\end{array}
\label{7.66}
\end{equation}%
with the help of $\phi (\nabla _{\xi }h)\phi X=-\phi ^{2}(\nabla _{\xi }h)X.$
Combining Eqs. (\ref{7.65}) and (\ref{7.66}) we deduce 
\begin{equation*}
Q\phi X+\phi QX=-2\left[ \alpha ^{2}\phi +\phi h^{2}\right] X-2tr(l)\phi
X+r\phi X.
\end{equation*}%
Then substituting Eq. (\ref{7.56}) in the last equation and using Eq. (\ref%
{7.63}), we obtain%
\begin{equation*}
Q\phi X+\phi QX=-tr(l)\phi X+r\phi X.
\end{equation*}%
By virtue of Eqs. (\ref{7.65}) (\ref{7.66}) and (\ref{7.63}), we also obtain
that%
\begin{equation}
(\nabla _{\xi }h)X=\frac{1}{2}f_{1}h\phi X+\frac{1}{2}(f_{2}-4\alpha )hX.
\label{7.67}
\end{equation}%
Using Eq. (\ref{7.67}) in Eq. (\ref{7.58}), we have%
\begin{equation}
QX=\tilde{a}X+\tilde{b}\eta (X)\xi +2\alpha \phi hX+\frac{1}{2}f_{1}hX+\frac{%
1}{2}(f_{2}-4\alpha )\phi hX,  \label{7.68}
\end{equation}%
for $\sigma \equiv 0$. Finally, substituting Eq. (\ref{7.68}) in Eq. (\ref%
{7.70}), we deduce%
\begin{eqnarray*}
R(X,Y)\xi &=&(tr(l)+\tilde{a}-\frac{r}{2})\left[ \eta (Y)X-\eta (X)Y\right] +%
\frac{1}{2}f_{1}\left[ \eta (Y)hX-\eta (X)hY\right] \\
&&+\frac{1}{2}f_{2}\left[ \eta (Y)\phi hX-\eta (X)\phi hY\right] ,
\end{eqnarray*}%
where $\tilde{a}=\frac{1}{2}r+\alpha ^{2}+\lambda ^{2}.$ Thus this shows
that $(M^{3},\phi ,\xi ,\eta ,g)$ is a $(\kappa ,\mu ,\nu )$-space.
\end{proof}

An interesting question is, Do there exist almost $\alpha $-cosymplectic
manifolds satisfying (\ref{501}) with $\kappa ,\mu $ non-constant smooth
functions? Now, we construct an example in order to answer this question for
the three dimensional case.

\begin{example}
Consider the three dimensional manifold%
\begin{equation*}
M^{3}=\left\{ (x,y,z)\in \mathbb{R}^{3},\text{ \ }z\neq 0\right\} ,
\end{equation*}%
where $(x,y,z)$ are the Cartesian coordinates in $\mathbb{R}^{3}.$ We define
three vector fields on $M$ as 
\begin{equation*}
\begin{array}{l}
e=\dfrac{\partial }{\partial x},~\phi e=\dfrac{\partial }{\partial y}, \\ 
\xi =\left[ \alpha x-y(e^{-2\alpha z}+z)\right] \dfrac{\partial }{\partial x}
\\ 
+\left[ x(z-e^{-2\alpha z})+\alpha y\right] \dfrac{\partial }{\partial y}+%
\dfrac{\partial }{\partial z},%
\end{array}%
\end{equation*}
where $\alpha $ is a real number. We easily get 
\begin{equation*}
\begin{array}{lll}
\lbrack e,\phi e] & = & 0, \\ 
\lbrack e,\xi ] & = & \alpha e+(z-e^{-2\alpha z})\phi e, \\ 
\lbrack \phi e,\xi ] & = & -(e^{-2\alpha z}+z)e+\alpha \phi e.%
\end{array}%
\end{equation*}

Moreover, the matrice form of the metric tensor $g$, the tensor fields $\phi 
$ and $h$ are given by%
\begin{equation*}
g=\left( 
\begin{array}{lll}
1 & 0 & -d \\ 
0 & 1 & -k \\ 
-d & -k & 1+d^{2}+k^{2}%
\end{array}%
\right) ,
\end{equation*}%
and%
\begin{equation*}
\phi =\left( 
\begin{array}{ccc}
0 & -d & k \\ 
1 & 0 & -d \\ 
0 & 0 & 0%
\end{array}%
\right) ,\text{ }h=\left( 
\begin{array}{ccc}
e^{-2z} & 0 & -de^{-2z} \\ 
0 & -e^{-2z} & ke^{-2z} \\ 
0 & 0 & 0%
\end{array}%
\right) ,
\end{equation*}%
where \ 
\begin{equation*}
\begin{array}{lll}
d & = & \alpha x-y(e^{-2\alpha z}+z), \\ 
k & = & x(z-e^{-2\alpha z})+\alpha y.%
\end{array}%
\end{equation*}%
Let $\eta $ be the $1$-form defined by $\eta =k_{1}dx+k_{2}dy+k_{3}dz$ for
all vector fields on $M^{3}.$ Since $\eta (X)=g(X,\xi ),$ we can easily
obtain that $\eta (e)=0,$ $\eta (\phi e)=0$ and $\eta (\xi )=1.$ By using
these equations, we get $\eta =dz$ for all vector fields. Since $d\eta
=d(dz)=d^{2}z,$ we obtain $d\eta =0.$ Using Koszul's formula , we have seen
that $d\Phi =2\alpha \eta \wedge \Phi .$ Hence, it has been showed that $%
M^{3}$ is an almost $\alpha $-cosymplectic manifold. Thus we obtain%
\begin{equation*}
R(X,Y)\xi =-(e^{-4\alpha z}+\alpha ^{2})\left[ \eta (Y)X-\eta (X)Y\right] +2z%
\left[ \eta (Y)hX-\eta (X)hY\right] ,
\end{equation*}%
where $\kappa =-(e^{-4\alpha z}+\alpha ^{2})$ and $\mu =2z$.
\end{example}


\begin{thebibliography}{99}
\bibitem{kimpak} T. W. Kim, H. K. Pak, Canonical foliations of certain
classes of almost contact metric structures, Acta Math. Sinica, Eng. Ser.
Aug., 21, 4 (2005), 841--846.

\bibitem{pastore} G. Dileo, A. M. Pastore, Almost Kenmotsu manifolds and
local symmetry, Bull. Belg. Math. Soc. Simon Stevin, 14 (2007), 343--354.

\bibitem{boeckx} E. Boeckx, J. T. Cho, $\eta $-parallel contact metric
spaces, Differential geometry and its applications, 22 (2005), 275--285.

\bibitem{blair} D. E., Blair, Riemannian geometry of contact and symplectic
manifolds, Progress in Mathematics, 203. Birkh\^{a}user Boston, Inc.,
Boston, MA, (2002).

\bibitem{vaisman} I. Vaisman, Conformal changes of almost contact metric
manifolds, Lecture Notes in Math., Berlin-Heidelberg-New York, 792 (1980),
435--443.

\bibitem{kenmotsu} K. Kenmotsu, A class of contact Riemannian manifold,
Tohoku Math. Journal, 24 (1972), 93--103.

\bibitem{olszak} Z. Olszak, Locally conformal almost cosymplectic manifolds,
Coll. Math., 57 (1989), 73--87.

\bibitem{yano} K. Yano, M. Kon, Structures on manifolds, Series in Pure
Mathematics, 3. World Scientific Publishing Co., Singapore, (1984).

\bibitem{ghosh} A. Ghosh, R. Sharma, J. T. Cho, Contact metric manifolds
with $\eta $-parallel torsion tensor, Ann. Glob. Anal. Geom., DOI
10.1007/s10455-008-9112-1., (2008).

\bibitem{kimpak1} T. W. Kim, H. K. Pak, Criticality of characteristic vector
fields on almost cosymplectic manifolds, J. Korean Math. Soc., 44, 3,
(2007), 605-613.

\bibitem{dacko olzsak} P. Dacko and Z.Olszak, On almost cosymplectic $(-1,%
%TCIMACRO{\U{3bc} }%
%BeginExpansion
\mu
%EndExpansion
,0)$-spaces. Cent. Eur. J. Math. 3 (2005), no. 2, 318--330 .

\bibitem{dacko} P. Dacko and Z.Olszak, On conformally flat almost
cosymplectic manifolds with Keahlerian leaves, Rend. Sem. Mat. Univ. Pol.
Torino, Vol. 56, 1(1998), 89-103.

\bibitem{pastore2} G. Dileo, A. M. Pastore, Almost Kenmotsu manifolds with a
condition of $\eta $-parallelism, Differential Geometry and its
Applications, 27 (2009) 671--679.

\bibitem{pastore3} G. Dileo, A. M. Pastore, Almost Kenmotsu Manifolds and
Nullity Distributions, J. Geom. 93 (2009), 46--61.

\bibitem{bo1} E. Boeckx. A full classification of contact metric $(\kappa ,%
%TCIMACRO{\U{3bc} }%
%BeginExpansion
\mu
%EndExpansion
)$-spaces, Illinois J. Math., 44(1), \ 2000, 212--219.

\bibitem{blr} D. E. Blair, T. Koufogiorgos and B. J. Papantoniou, Contact
metric manifolds satisfying a nullity condition, Israel J. Math., 91(1995),
189--214.

\bibitem{olszak3} P. Dacko , Z. Olszak. On almost cosymplectic $(\kappa ,%
%TCIMACRO{\U{3bc} }%
%BeginExpansion
\mu
%EndExpansion
,\nu )$-spaces, Banach Center Publ., 69, 2005,211-220.
\end{thebibliography}
\end{document}